\documentclass{amsart}
\usepackage{amscd}
\usepackage{verbatim}
\usepackage{amssymb,amsmath}
\input xy
\xyoption{all}

\newtheorem{maintheorem}{Theorem}
\newtheorem{maincorollary}{Corollary}
\newtheorem{theorem}{Theorem}[section]
\newtheorem{lemma}[theorem]{Lemma}
\newtheorem{proposition}[theorem]{Proposition}
\newtheorem{corollary}[theorem]{Corollary}

\theoremstyle{definition}
\newtheorem{definition}[theorem]{Definition}

\theoremstyle{remark}
\newtheorem{remark}[theorem]{Remark}
\newtheorem{example}[theorem]{Example}

\makeatletter\let\c@equation=\c@theorem\makeatother

\numberwithin{equation}{section}

\newcommand{\Cat}{\mathit{Cat}}
\newcommand{\co}{\colon\thinspace}
\newcommand{\comm}{\mathit{comm}}

\newcommand{\colim}{\operatornamewithlimits{colim}}

\newcommand{\Fin}{\mathit{Fin}}

\newcommand{\hocolim}{\operatornamewithlimits{hocolim}}
\newcommand{\I}{\mathcal I}
\newcommand{\id}{\text{id}}

\renewcommand{\mod}{\mathit{mod}}

\newcommand{\Map}{\operatorname{Map}}

\newcommand{\Simp}{\mathit{Simp}}
\newcommand{\Sp}{\textit{Sp}}
\newcommand{\Spin}{\mathit{Spin}}
\newcommand{\Th}{\operatorname{TH}}

\newcommand{\U}{\mathcal U}

\newcommand{\xl}{\xleftarrow}
\newcommand{\xr}{\xrightarrow}

\begin{document}

\title{Higher topological Hochschild homology of Thom spectra}
\author{Christian Schlichtkrull}
\address{Department of Mathematics, University of Bergen, Johannes
  Brunsgate 12, 5008 Bergen, Norway}
\email{krull@math.uib.no}
\date{\today}

\begin{abstract}
In this paper we analyze the higher topological Hochschild homology of commutative Thom $S$-algebras. This includes the case of the classical cobordism spectra $MO$, $MSO$, $MU$, etc. We consider the homotopy orbits of the torus action on iterated topological Hochschild homology and we describe the relationship to topological Andr\'{e}-Quillen homology.  
\end{abstract}

\maketitle

\section{introduction}
The simplicial model of the Hochschild homology complex associated to a commutative ring $T$ and a $T$-module $M$ was extended by Loday \cite{Loday89, Loday} to a functor $\mathcal L(T,M)$ from the category of based simplicial sets to the category of simplicial $T$-modules. Evaluated at the simplicial circle $S^1$ this construction gives the Hochschild complex and when applied to the simplicial $n$-spheres the outcome is the higher order Hochschild homology analyzed by Pirashvili \cite{Pi}. These higher order groups stabilize to give the $\Gamma$-homology groups of Robinson-Whitehouse 
\cite{PiRi, RoWh} which in turn are isomorphic to the topological Andr\'{e}-Quillen homology of the associated Eilenberg-Mac Lane spectra, see  \cite{BaMc}.

More recently, Brun, Carlsson, and Dundas \cite{BCD} have introduced a topological version of this whose input is a commutative $S$-algebra $T$ (that is, a commutative symmetric ring spectrum) and a $T$-module $M$. As we recall in Section \ref{LodayS}, the topological version of the Loday construction is a functor $\mathcal L(T,M)$ from the category of based simplicial sets to the category of $T$-modules whose value at the simplical circle is the topological Hochschild homology spectrum $\Th(T,M)$. 
Applied to the simplicial $n$-spheres this construction gives the topological analogue of Pirashvili's higher order Hochschild homology and evaluating the functor $\mathcal L(T,T)$ on the simplicial $n$-torus $(S^1)^n$ the outcome is the $n$-fold iterated topological Hochschild homology spectrum 
$\Th^{(n)}\!(T)$. In this paper we use the term \emph{higher topological Hochschild homology} in a broad sense to mean the outcome obtained by applying the functor $\mathcal L(T,M)$. The interest in this construction is motivated in part by its relation to algebraic K-theory and chromatic phenomena in homotopy theory, see \cite{CDD}. In particular, the action of the topological $n$-torus $\mathbb T^n$ on $\Th^{(n)}\!(T)$ leads to a higher version of topological cyclic homology which is expected to be a strong invariant of iterated algebraic K-theory. 

The purpose of this paper is to analyze the functor $\mathcal L(T,M)$ in the case where $T$ is a commutative Thom $S$-algebra. Let $BF_{h\I}$ be the model of the classifying space for stable spherical fibrations introduced in \cite{Sch2}. This is an $E_{\infty}$ space (it has an action of the Barratt-Eccles operad) and, as we recall in Section \ref{symmetricthomsection}, a map of 
$E_{\infty}$ spaces 
$f\co A\to BF_{h\I}$  gives rise to a commutative Thom $S$-algebra $T(f)$. In particular, when 
$f$ is the map associated to one of the stabilized Lie groups $O$, $SO$, $U$, $\Spin$, \dots, this construction gives a model of the associated cobordism spectrum $MO$, $MSO$, \dots, as a commutative $S$-algebra. (The approach to Thom spectra described here is a symmetric spectrum analogue of the Lewis-May Thom spectrum functor from \cite{LMS}). 
Let us write $\mathbf A$ for the spectrum associated to the 
$E_{\infty}$ space  $A$. Thus, $\mathbf A$ is a positive $\Omega$-spectrum whose $0$th space equals $A$ and the canonical map $A\to \Omega^{\infty}\mathbf A$ from $A$ to the associated infinite loop space is a group completion (see Section \ref{Iinfinite} for details). By definition, $A$ is grouplike if this map is a weak equivalence. Given a based simplicial set $X$ we write $\mathbf A\wedge X$ for the smash product of $\mathbf A$ with the realization of $X$ and we write 
$\Omega^{\infty}(\mathbf A\wedge X)$ for the associated infinite loop space. The following theorem is our main result.  

\begin{maintheorem}\label{Einftyhighthom}
Let $f\co A\to BF_{h\I}$ be a map of $E_{\infty}$ spaces with associated commutative Thom $S$-algebra $T(f)$ and let $M$ be a $T(f)$-module. 
If $A$ is grouplike, then there is a chain of natural stable equivalences of $S$-modules
$$
\mathcal L_X(T(f),M)\simeq M\wedge \Omega^{\infty}(\mathbf A\wedge X)_+
$$
for any based simplicial set $X$.
If $M$ is a commutative  $T(f)$-algebra, then this is an equivalence of $E_{\infty}$ $S$-algebras. 
\end{maintheorem}

Here and in the following we implicitly consider the \emph{derived} version of $\mathcal L_X(T,M)$ when making statements about its homotopy type. That is,  we implicitly assume that $T(f)$ has been replaced by a cofibrant commutative $S$-algebra and that $M$ has been replaced by a cofibrant module. 
For the last statement in the theorem we view the infinite loop space $\Omega^{\infty}(\mathbf A\wedge X)$ as an $E_{\infty}$ space and $M\wedge \Omega^{\infty}(\mathbf A\wedge X)_+$ has the induced $E_{\infty}$ action. We shall give a more precise formulation of this in Theorem 
\ref{Ispacehighthom} where we describe the $E_{\infty}$ structure in terms of the $\Gamma$-space (in the sense of Segal \cite{Se}) associated to $A$. The simplest instance of Theorem 
\ref{Einftyhighthom} is the case $X=S^0$ where we get the stable equivalence
$$
M\wedge T(f)\simeq M\wedge A_+.
$$
For $T(f)$ oriented (that is, $f$ factors over $BSF_{h\I}$) and $M$ the integral Eilenberg-Mac Lane spectrum this is the homotopy theoretical version of the Thom isomorphism theorem. In this way one may view Theorem \ref{Einftyhighthom} as a generalized Thom isomorphism theorem for commutative $S$-algebras. 
We now consider some immediate corollaries. Let in the following $f\co A\to BF_{h\I}$ be a map of $E_{\infty}$-spaces with associated commutative Thom $S$-algebra $T(f)$. Consider first the case of the simplicial circle 
$X=S^1$ and notice that $\Omega^{\infty}(\mathbf A\wedge S^1)$ may be identified with the first space $BA$ in the spectrum $\mathbf A$. Hence we recover the calculation of topological Hochschild homology from \cite{BCS}. 

\begin{maincorollary}
If $A$ is grouplike, then there is a stable equivalence of $S$-modules
$$
\Th(T(f),M)\simeq M\wedge BA_+
$$
which is an equivalence of $E_{\infty}$ $S$-algebras if $M$ is a commutative $T$-algebra.
\end{maincorollary} 

This result complements that of \cite{BCS} in that we here keep track of the $E_{\infty}$ structure when $M$ is a commutative $T$-algebra. A similar result has been obtained independently by Blumberg \cite{Bl} using a different method. Working in the $S$-module framework from \cite{EKMM}, Blumberg proves more generally that the Thom spectrum functor preserves indexed colimits which is closely related to the statement in Theorem \ref{Einftyhighthom}. It should be emphasized that the main result in \cite{BCS} is a calculation of the topological Hochschild homology of any (not necessarily commutative) Thom $S$-algebra. This involves an interesting action of the stable Hopf map which is not visible in the commutative case. 

Letting $T=M$ and applying Theorem \ref{Einftyhighthom} to the simplicial n-torus 
we get a calculation of iterated topological Hochschild homology. 

\begin{maincorollary}
If $A$ is grouplike, then 
there is a natural stable equivalence of $E_{\infty}$ $S$-algebras
$$
\Th^{(n)}\!(T(f))\simeq T(f)\wedge\Omega^{\infty}(\mathbf A\wedge \mathbb T^n)_+.
$$
\end{maincorollary} 

Using the subdivision techniques from \cite{McC}  this in particular identifies the action of the Adams operations (induced by the power maps of $\mathbb T^n$). However, it is clear that this equivalence does not reflect the $\mathbb T^n$-equivariant structure of $\Th^{(n)}\!(T(f))$. In order to get information about the latter we instead consider the based simplicial set
$(S^1)^n_+$ obtained by adjoining a disjoint base point. For $A$ grouplike it follows from   
Theorem \ref{Einftyhighthom} that there is a stable equivalence
\begin{equation} \label{equivariantstable}
M\wedge\Th^{(n)}\!(T(f))\simeq \mathcal L_{(S^1)^n_+}(T(f),M)\simeq M\wedge 
\Omega^{\infty}(\mathbf A\wedge \mathbb T^n_+)_+
\end{equation}
which is now $\mathbb T^n$-equivariant. We therefore get an induced equivalence 
of homotopy orbit spectra for any compact subgroup of $\mathbb T^n$. Letting $M$ be the integral Eilenberg-Mac Lane spectrum we show in Section \ref{homotopyorbit} that this leads to the following conclusion about the integral spectrum homology groups.

\begin{maincorollary}\label{homotopyorbitcorollary}
If $A$ is grouplike and $T(f)$ oriented there is an isomorphism
$$
H_*(\Th^{(n)}\!(T(f))_{hG})\simeq H_*(\Map(\mathbb T^n,B^nA)_{hG})
$$
for any compact subgroup $G$ of $\mathbb T^n$.
\end{maincorollary}

Here the subscript $(-)_{hG}$ on the left hand side indicates the homotopy orbit spectrum and the right hand side is the classical integral  Borel homology of the mapping space $\Map(\mathbb T^n,B^nA)$ with $\mathbb T^n$ acting on the domain. Spelling this out for  $MU$ in the case $n=1$ we get that
$$
H_*(\Th(MU)_{hG})\simeq H_*(\Map(\mathbb T,SU)_{hG})
$$
for any compact subgroup $G$ of $\mathbb T$.
These homotopy orbit spectra are important since they interpolate between the fixed point spectra via the fundamental cofibration sequences, see \cite{HM}. For instance, we have the cofibration sequence 
$$
\Th(MU)_{hC_{p^n}}\to \Th(MU)^{C_{p^n}}\to \Th(MU)^{C_{p^{n-1}}}
$$
where $p$ is a prime number and $C_{p^n}$ denotes the cyclic group of order $p^n$. It is proved in \cite{BCD} that there are analogous cofibration sequences relating the fixed points of iterated  topological Hochschild homology. Thus, even though Theorem~\ref{Einftyhighthom} does not extend directly to a calculation of the fixed point spectra our work does give the initial input for such a calculation. We also remark that the stable equivalence (\ref{equivariantstable}) leads to a description of the continuous homology of the homotopy fixed point spectra 
$\Th^{(n)}\!(T(f))^{hG}$. 
This was pointed out to us by S.\ Lun\o e-Nielsen and J.~Rognes. 

As we recall in Section \ref{LodayS}, the Loday construction $\mathcal L_X(T,T)$ may be identified with the tensor product $X\otimes T$ in the category of commutative $S$-algebras. In the case $X=S^n$ we deduce from Theorem \ref{Einftyhighthom} the stable equivalence
\begin{equation}\label{S^nequivalence}
S^n\otimes T(f)\simeq T(f)\wedge B^nA_+.
\end{equation}
When $f$ is the map $BG\to BF$ associated to a stabilized Lie group $G$, the above stable equivalence has the following cobordism interpretation. 

\begin{maincorollary}
Let $G$ denote one of the stabilized Lie groups $O$, $SO$, $U$, $\Spin$, \dots, and let 
$\Omega^G_*(-)$ be the associated cobordism theory. Then there is an isomorphism of graded rings
$$
\pi_*(S^n\otimes MG)\simeq \Omega^G_*(B^{n+1}G).
$$ 
\end{maincorollary} 

Working in the category of $S$-algebras introduced in \cite{EKMM}, Basterra defines in \cite{Ba} the \emph{topological cotangent complex} $\mathbf{L}\mathbf{\Omega}_RT$ associated to a commutative $S$-algebra $R$ and a commutative $R$-algebra $T$. This is a $T$-module whose homotopy groups are the topological Andr\'{e}-Quillen homology groups of $T$. In the following we shall only be interested in the case where $R$ equals $S$. We define the reduced tensor product $S^n\tilde\otimes T$ to be the cofiber of the map $T\to S^n\otimes T$ induced by the inclusion of the base point of $S^n$. There are canonical stabilization maps
$$
S^1\wedge(S^n\tilde\otimes T)\to S^{n+1}\tilde\otimes T
$$
and it follows from the discussion in \cite{BM}, Section 3, that 
$\mathbf{L}\mathbf{\Omega}_ST$ may be identified with the stabilization 
$\hocolim_n\Omega^n(S^n\tilde\otimes T)$. Here we implicitly choose a cofibrant replacement of $T$ when defining $S^n\tilde\otimes T$. Even though Basterra-Mandell works in the framework from \cite{EKMM} the homotopy invariant nature of this expression makes it valid also with our interpretation of $S$-algebras (that is, symmetric ring spectra). In the case of the commutative Thom $S$-algebra $T(f)$ associated to an  $E_{\infty}$ map $f\co A\to BF$ with $A$ grouplike, we see from (\ref{S^nequivalence}) that 
$S^n\tilde\otimes T(f)$ can be identified with $T(f)\wedge B^nA$. Furthermore, since the spectra $T(f)$ and $\mathbf A$ are connective, there is a stable equivalence
$$
\hocolim_{n}\Omega^n(T(f)\wedge B^nA)\simeq T(f)\wedge \mathbf A.
$$
Hence we recover the following result of Basterra--Mandell.

\begin{maincorollary}[\cite{BM}]
If $A$ is grouplike, then 
$\mathbf{L}\mathbf{\Omega}_ST(f)$ is stably equivalent to $T(f)\wedge\mathbf A$.
\end{maincorollary}

In the above we have implicitly replaced $T$ and $M$ by cofibrant objects when discussing the homotopy type of $\mathcal L_X(T,M)$. However, for many purposes it is more convenient to work with an explicit model of this functor which is homotopy invariant in $T$ and $M$ without cofibrancy assumptions. Thus, a major part of the technical work in this paper goes into the construction of such a homotopy invariant model $\Lambda_X(T,M)$. (The only requirement on $T$ and $M$ for this to represent the correct ``derived'' homotopy type of $\mathcal L_X(T,M)$ is that they be levelwise well-based). We discuss the technical difficulties encountered if one tries to prove Theorem \ref{Einftyhighthom} using the Loday construction directly in the beginning of Section \ref{highthomsection}. A further motivation for introducing the functor $\Lambda_X(T,M)$ is that it naturally leads to a definition of higher topological cyclic homology. We shall return to this in a separate paper and include a brief discussion of the relevant issues in Section \ref{homotopyorbit}. (Another and technically quite different homotopy invariant model which leads to a good equivariant theory has been introduced in \cite{BCD}). The functor 
$\Lambda_X(T,M)$ is defined as a certain homotopy colimit and consequently does not take values in strictly commutative $S$-algebras when $M$ is a commutative $T$-algebra. Fortunately, this is not a serious limitation since in this case it naturally takes values in $E_{\infty}$ $S$-algebras instead. We present a general framework for analyzing multiplicative properties of homotopy colimits in  
Appendix \ref{appendix}. The main result here is that in a precise sense homotopy colimits take $E_{\infty}$ diagrams to $E_{\infty}$ 
$S$-algebras. 

\subsection*{Organization of the paper}
We begin by introducing the Loday construction $\mathcal L_X(T,M)$ and the homotopy invariant version $\Lambda_X(T,M)$ in Section \ref{highersection}. Here we also discuss the multiplicative properties of $\Lambda_X(T,M)$ when $M$ is a commutative $S$-algebra. In Section 
\ref{Isymmetricthomsection} we collect the necessary background material on $\I$-spaces and symmetric Thom spectra. Following this, in Section \ref{highthomsection}, we then prove the main results of the paper. In Appendix \ref{appendix} we discuss the multiplicative properties of homotopy colimits in general.   

\subsection*{Acknowledgments} The author would like to thank Morten Brun, Bj\o rn Dundas, Birgit Richter, and John Rognes for helpful conversations related to this work.

\subsection{Notation and conventions}\label{conventions}
Let $\mathcal U$ and $\mathcal T$ be the categories of unbased and based compactly generated weak Hausdorff spaces. We shall work in the category $\Sp^{\Sigma}$ of topological symmetric spectra introduced in \cite{MMSS}; see also \cite{HSS} for the simplicial version and \cite{Schwede}
 for a thorough introduction to symmetric spectra. Let $S$ be the sphere spectrum such that an \emph{$S$-module} is the same thing as a symmetric spectrum and an 
\emph{$S$-algebra} is a symmetric ring spectrum. There are several model structures on 
$\Sp^{\Sigma}$ but for our purposes it will be most convenient to work with the \emph{flat} model structure from \cite{Sh2} (here we use the terminology from \cite{Schwede}; this is what is called the $S$-model structure in \cite{Sh2}). This model structure gives rise to accompanying model structures on the categories of $S$-algebras, commutative $S$-algebras, and, for an $S$-algebra $T$, the category of $T$-modules. The facts about the flat module structure that are important for us is that (i) the underlying $S$-module of a cofibrant commutative $S$-algebra is cofibrant, see \cite{Sh2}, Corollary 4.3, and (ii) if $T$ is a cofibrant commutative $S$-algebra, then the underlying $S$-module of a cofibrant $T$-module is cofibrant. Since smash products with cofibrant $S$-modules preserve stable equivalences this ensures that the Loday construction 
$\mathcal L(T,M)$ is a homotopy functor in $T$ and $M$ when restricted to cofibrant objects, see Section \ref{LodayS}. 

An $S$-module $M$ is said to be \emph{connective} if its spectrum homotopy groups $\pi_*M$ vanish in negative degrees and \emph{convergent} if there exists an unbounded, non-decreasing sequence of integers $\lambda_n$ such that each of the adjoint structure maps $M_n\to \Omega M_{n+1}$ is ($\lambda_n+n$)-connected. It will sometimes be necessary to assume that an $S$-module is \emph{well-based} in the sense that each of its spaces has a non-degenerate base point.  By an $E_{\infty}$ object in $\mathcal U$ or $\Sp^{\Sigma}$ we understand an object with an action of the Barratt-Eccles operad $\mathcal E$, see Section \ref{hocolimEinfty}. There are standard change of operad techniques, \cite{BCS}, \cite{May}, for comparing actions by different operads, but it is the Barratt-Eccles operad that occurs most naturally in our setting.

\subsection{Homotopy colimits of $S$-modules}\label{hocolimsection}
Homotopy colimits of $S$-modules are defined in analogy with homotopy colimits of based spaces, see \cite{BK}. Thus, if $\mathcal A$ is a small category and $F\co \mathcal A\to \Sp^{\Sigma}$ a functor, then $\hocolim_{\mathcal A}F$ is the realization of the simplicial
$S$-module
\begin{equation}\label{hocolimdefinition}
[k]\mapsto \bigvee_{a_0\leftarrow\dots\leftarrow a_k}F(a_k),
\end{equation}
where the wedge product is over all $k$-tuples of composable arrows in $\mathcal A$. Since realization of simplicial $S$-modules is performed levelwise, this is the same thing as the 
levelwise homotopy colimit of $F$. The functorial properties of the homotopy colimit is best understood by introducing the category $\mathcal D(\Sp^{\Sigma})$ of diagrams in 
$\Sp^{\Sigma}$. An object of $\mathcal D(\Sp^{\Sigma})$ is a pair $(\mathcal A,F)$, where 
$\mathcal A$ is a small category and $F\co \mathcal A\to \Sp^{\Sigma}$ is a functor as above. A morphism from $(\mathcal A,F)$ to $(\mathcal A',F')$ is a pair 
$(\varepsilon,\phi)$ whose first component is a functor $\varepsilon\co \mathcal A\to \mathcal A'$ and whose second component is a natural transformation $\phi\co F\to F'\circ \varepsilon$. The homotopy colimit defines a functor $\mathcal D(\Sp^{\Sigma})\to \Sp^{\Sigma}$ in which a morphism $(\varepsilon,\phi)$ as above gives rise to the map
$$
\hocolim_{\mathcal A}F\xr{\phi}\hocolim_{\mathcal A}F'\circ\varepsilon\to 
\hocolim_{\mathcal A'}F'.
$$
Here the last map is defined by applying $\varepsilon$ to the indexing category. 
For any category $\mathcal C$, the data defining a functor 
$\mathcal C\to \mathcal D(\Sp^{\Sigma})$ amounts to 
\begin{itemize}
\item
a functor $\Phi\co\mathcal C\to\Cat$ from the category $\mathcal C$ to the category $\Cat$ of small categories,
\item
for each object $c$ in $\mathcal C$ a functor $F_c\co\Phi(c)\to\Sp^{\Sigma}$,
\item
for each morphism $f\co c\to d$ in $\mathcal C$ a natural transformation
$$
\phi_f\co F_c\to F_d\circ\Phi(f).
$$
\end{itemize}
It is required that $\phi_{1_{\mathcal C}}$ be the identity natural transformation on $F_c$ and that for each composable pair of morphisms $f\co c\to d$ and $g\co d\to e$, the diagram
$$
\begin{CD}
F_c@>\phi_f>> F_d\circ\Phi(f)\\
@VV \phi_{g\circ f} V @VV \phi_g\circ \Phi(f) V\\
F_e\circ \Phi(g\circ f)@= F_e\circ \Phi(g)\circ\Phi(f)
\end{CD}
$$
is commutative. 

\section{Higher topological Hochschild homology}\label{highersection}
In this section we first recall the definition of the Loday functor and we then construct the homotopy invariant model $\Lambda_X(T,M)$.

\subsection{The Loday construction for commutative $S$-algebras}\label{LodayS}
For a (discrete) commutative ring $T$ and a $T$-module $M$, the Loday construction
$\mathcal L(T,M)$ is the functor that to a finite based set $X$ with base point $*$ associates the $T$-module $\bigotimes _{x\in X}T_x$, where $T_x=T$ for $x\neq *$ and $T_*=M$. A map of finite based sets $f\co X\to Y$ gives rise to a map of $T$-modules defined by
$$
\bigotimes _{x\in X}T_x\xr{\sim} \bigotimes_{y\in Y}
\,
\bigotimes_{x\in f^{-1}(y)}T_x\to 
\bigotimes_{y\in Y}T_y,
$$
where the first map permutes the tensor factors and the second map is the tensor product of the maps  $\bigotimes_{f^{-1}(y)}T_x\to T_y$ induced by the multiplication in $T$ and, for $y=*$, the $T$-module structure on $M$. A similar functor can be defined for a commutative monoid in any symmetric monoidal category; in particular in the category of $S$-modules. Thus, for a commutative $S$-algebra $T$ and a $T$-module $M$, we have the functor $\mathcal L(T,M)$ that to a finite based set $X$ associates the smash product $\bigwedge_{x\in X}T_x$, where we again use the notation $T_x=T$ for $x\neq *$ and $T_*=M$. Applying the usual levelwise realization of simplicial $T$-modules, this construction extends to a functor 
$$
\mathcal L(T,M)\co \Simp_*\to \text{$T$-$\mod$}
$$ 
from the category $\Simp_*$ of based simplicial sets to the category $T$-$\mod$ of $T$-modules. In detail, given a (not necessarily finite) based set $X$, we let $\mathcal L_X(T,M)$ be the colimit
$$
\mathcal L_X(T,M)=\colim_{Y\subseteq X}\mathcal L_Y(T,M)
$$ 
over the finite based subsets of $X$ and for a based simplicial set, again denoted $X$, we define 
$\mathcal L_X(T,M)$ to be the realization of the simplicial $T$-module obtained by applying this construction in each simplicial degree. 

Depending on $M$, the construction $\mathcal L(T,M)$ has added functoriality. Suppose first that $M$ is a commutative $T$-algebra. Then $\mathcal L_X(T,M)$ inherits the structure of a commutative $S$-algebra and the canonical map $T\to\mathcal L_X(T,M)$ makes it a $T$-algebra. In the case where $M$ is the $S$-algebra $T$ itself, the functor $\mathcal L(T,T)$ extends to the category 
$\Simp$ of (unbased) simplicial sets in the sense that there is a functor $\mathcal L(T)$ that makes the diagram
$$
\begin{CD}
\Simp_*@>\mathcal L(T,T)>> \text{$T$-$\comm$}\\
@VVV @VVV\\
\Simp@>\mathcal L(T)>>\text{$S$-$\comm$} 
\end{CD}
$$  
commutative. Here $S$-$\comm$ and $T$-$\comm$ denote the categories of commutative 
$S$- and $T$-algebras respectively, and the vertical arrows represent the obvious forgetful functors.

\begin{remark}\label{tensorremark}
The functor $\mathcal L(T)$ can be interpreted in terms of the tensor structure on the category of commutative $S$-algebras with respect to (unbased) simplicial sets: given an (unbased) simplicial set $X$, the commutative $S$-algebra $\mathcal L_X(T)$ is isomorphic to the tensor product 
$X\otimes T$. For a based simplicial set $X$ we then have the natural isomorphism
\[
\mathcal L_X(T,M)\cong M\wedge_T \mathcal L_X(T,T)
\cong M\wedge_T(X\otimes T)
\]  
for any $T$-module $M$. We refer the reader to \cite{Bo} for a general introduction to enriched category theory (which we shall not use).
\end{remark}

 \subsection{B\"okstedt's construction}\label{Bsection}
 Whereas the Loday construction $\mathcal L_X(T,M)$ is homotopy invariant only when the input variables $T$ and $M$ are cofibrant, we aim to construct a model which is homotopy invariant for all $T$ and $M$ as long as these are well-based. We first recall  B\"okstedt's homotopy invariant model of $\mathcal L_X(T,M)$ when $X$ is a finite based set. Let 
 $\I$ be the category whose objects are the finite sets $\mathbf n=\{1,\dots,n\}$, including the empty set $\mathbf 0$, and whose morphisms are the injective maps. The concatenation  
 $\mathbf m\sqcup\mathbf n$ of two objects in $\I$ is defined to be the object $\{1,\dots,m+n\}$ with $\mathbf m$ corresponding to the first $m$ letters and $\mathbf n$ corresponding to the last $n$ letters. This operation makes $\I$ a symmetric monoidal category. For a finite set $X$, not necessarily based, we write $\I(X)$ for the product category $\I^X$ and think of the objects as functors $j\co X\to\I$ from the discrete category $X$ to $\I$. Given based spaces $A$ and $B$, let $F_S(A,B)$ be the $S$-module with $n$th space
$$
 F_S(A,B)(n)=\Map_*(A,B\wedge S^n).
$$
Thus, $F_S(A,B)$ may be identified with the symmetric spectrum of  
$S$-module maps from the suspension spectrum of $A$ to the suspension spectrum of $B$.
Suppose now that $X$ is a finite set with a specified base point and let $j$ be an object in 
$\I(X)$. Then we define $F_X(T,M)(j)$ to be the $S$-module   
$$
F_X(T,M)(j)=F_S(\bigwedge_{x\in X}S(j_x),\bigwedge_{x\in X}T_x(j_x))
$$ 
where the notation $T_x$ is as in the definition of the Loday construction. 
Using the structure maps of the $S$-modules $T$ and $M$ this construction defines a functor 
$F_X(T,M)\co \I(X)\to \Sp^{\Sigma}$. Notice, that if $T$ and $M$ are well-based, then the values of the $\I(X)$-diagram $F_X(T,M)$ are also well-based. Using this, the following homotopy invariance result can be deduced from the simplicial analogue proved in \cite{Sh}, 
Corollary 4.2.9.

\begin{proposition}\label{Binvariant}
 Let $T\to T'$ be a stable equivalence of commutative $S$-algebras, $M$ a $T$-module, $M'$ a $T'$-module, and $M\to M'$ a stable equivalence of $T$-modules. Suppose further that all these $S$-modules are well-based. Then the induced map of homotopy colimits
$$
\hocolim_{\I(X)}F_X(T,M)\to \hocolim_{\I(X)}F_X(T',M') 
$$
 is a stable equivalence for any finite based set $X$. \qed  
\end{proposition} 
 
Given an $S$-module $M$, we write $F_*(M)$ for the $\I$-diagram $F_*(S,M)$ associated to the one-point set $*$. The homotopy colimit of this diagram gives rise to the ``detection functor'' $M\mapsto DM$ from \cite{Sh} and it follows from the proof of \cite{Sh}, Proposition 3.1.9 (3), that if $M$ is well-based, then there is a chain of stable equivalences
\begin{equation}\label{shipleyequi}
DM=\hocolim_{\I}F_*(M)\simeq M.
\end{equation}
This applies in particular to the well-based $S$-module $\mathcal L_X(T,M)$ associated to a cofibrant commutative $S$-algebra $T$ and a cofibrant $T$-module $M$. If we choose an ordering of $X$ there also is a natural map 
\begin{equation}\label{shipleyequi2}
\hocolim_{\I(X)}F_X(T,M)\to \hocolim_{\I}F_*(\mathcal L_X(T,M))
\end{equation}
induced by the concatenation functor $\I(X)\to \I$. It again follows from the simplicial analogue \cite{Sh}, Proposition 4.2.3, that this is a stable equivalence if $T$ and $M$ are cofibrant. Summarizing the above discussion we have the following. 

\begin{theorem}[\cite{Sh}]
Let $X$ be a finite based set with a chosen ordering. If $T$ is a cofibrant commutative $S$-algebra and $M$ is a cofibrant $T$-module, then there is a chain of natural stable equivalences
$$
\mathcal L_X(T,M)\simeq\hocolim_{\I(X)}F_X(T,M). 
\eqno\qed
$$
\end{theorem} 

Combining this result with Proposition \ref{Binvariant}, it follows that if $T$ and $M$ are well-based, then  the homotopy colimit of $F_X(T,M)$ has the stable homotopy type of the derived Loday construction $\mathcal L_X(T^c,M^c)$, where $T^c$ is a cofibrant replacement of $T$ and $M^c$ is a cofibrant replacement of $M$ as a $T^c$-module. 

\begin{remark}\label{Bokstedtremark}
B\"okstedt \cite{Bok} originally defined topological Hochschild homology by applying the above construction degree-wise to the simplicial circle. However, the fact that this gives a simplicial spectrum is special to the circle and this construction does not work for simplicial sets in general. The problem is that the categories $\I(X)$ are not functorial in $X$ which in turn is due to the fact that the symmetric monoidal structure of $\I$ is not strictly symmetric. 
\end{remark}

\subsection{The homotopy invariant model}\label{invariantsection}
As discussed in Remark \ref{Bokstedtremark}, the main problem with B\"okstedt's construction is the lack of functoriality in $X$. In order to remedy this we first replace the categories $\I(X)$ by equivalent categories $\I\langle X\rangle$ which have the advantage of being functorial in the (unbased) set $X$. Let $\mathcal P(X)$ be the category of subsets and inclusions in $X$. The objects of $\I\langle X\rangle$ are the functors 
$\theta\co \mathcal P(X)\to\I$ which satisfy the condition that for each pair of disjoint subsets 
$U$ and $V$, the diagram
$
\theta_U\to \theta_{U\cup V}\leftarrow \theta_V
$ 
is a coproduct diagram of finite sets (the category $\I$ itself does of course not have coproducts). Notice that such a functor takes the empty set  
to the initial object $\mathbf 0$. 
We view $\I\langle X\rangle$ as a full subcategory of the category of functors from 
$\mathcal P(X)$ to $\I$. 
Given a point $x\in X$, let $\theta_x$ be the value of the functor $\theta$ evaluated at 
$\{x\}$. For each ordering of a subset $U$ in $X$ we then get a canonical bijection
$
\textstyle{\coprod}_{x\in U}\theta_x\to \theta_U.
$
A map of sets $f\co X\to Y$ induces a functor $f_*\co \I\langle X\rangle \to \I\langle Y\rangle$ by letting $f_*\theta_U=\theta_{f^{-1}(U)}$ for each subset $U$ of $Y$. 
In this way we have defined a functor 
$$
\I\langle-\rangle\co \Fin\to\Cat,\quad X\mapsto \I\langle X\rangle
$$
from the category $\Fin$ of finite sets to the category of small categories. Consider the forgetful functor $\pi_X\co\I\langle X\rangle\to \I(X)$ that restricts an object $\theta$ to the one-point subsets in $X$. Arguing as in the analogous situation considered in \cite{Sch1}, Lemma 5.1, we get the following lemma. 
\begin{lemma}
Let $F\co \I(X)\to \Sp^{\Sigma}$ be a functor with values in well-based $S$-modules. Then the natural map
$$
\hocolim_{\I\langle X\rangle}F\circ\pi_X\to \hocolim_{\I(X)}F
$$
is a levelwise weak equivalence of $S$-modules. \qed
\end{lemma}

If $X$ comes equipped  with a specified base point, this lemma
in particular applies to the functor $F_X(T,M)$ associated to well-based $T$ and $M$. We claim that the correspondence 
$$
\Fin_*\to \mathcal D(\Sp^{\Sigma}),\quad X\mapsto (\I\langle X\rangle, F_X(T,M)\circ\pi_X)
$$
defines a functor from the category $\Fin_*$ of finite based sets to the category of diagrams in 
$\Sp^{\Sigma}$. Thus, for each based  map $f\co X\to Y$ we must exhibit a natural transformation
\begin{equation}\label{phitransformation}
\phi_f\co F_X(T,M)\circ\pi_X\to F_Y(T,M)\circ\pi_Y\circ f_*
\end{equation}
satisfying the conditions spelled out in Section \ref{hocolimsection}. 
The idea is that given an object $\theta$ in $\I\langle X\rangle$, the associated object $f_*\theta$ contains the data specifying how to multiply the smash factors indexed by $X$. Choosing an ordering of each of the subsets 
$f^{-1}(y)$, we define $d_f(\theta)$ to be the composition
$$
\bigwedge_{y\in Y}S(\theta_{f^{-1}(y)})\to\bigwedge_{y\in Y}S({\textstyle\coprod}_{x\in f^{-1}(y)}\theta_x)\to \bigwedge _{y\in Y}\bigwedge_{x\in f^{-1}(y)}S(\theta_x)\to\bigwedge_{x\in X}S(\theta_x)
$$ 
and $t_f(\theta)$ to be the composition
$$
\bigwedge_{x\in X}T_x(\theta_x)\to\bigwedge_{y\in Y}\bigwedge_{x\in f^{-1}(y)}T_x(\theta_x)\to\bigwedge_{y\in Y}T_y({\textstyle\coprod}_{x\in f^{-1}(y)}\theta_x)\to\bigwedge_{y\in Y}
T_y(\theta_{f^{-1}(y)}).
$$
Here the definition of the maps is supposed to be self-explanatory. The main point is that since $S$ and $T$ are commutative these maps do not depend on the orderings used to define them. For each $\theta$ we define the map
$$
\phi_f\co F_S(\bigwedge_{x\in X}S(\theta_x),\bigwedge_{x\in X}T_x(\theta_x))\to
 F_S(\bigwedge_{y\in Y}S(\theta_{f^{-1}(y)}),\bigwedge_{y\in Y}T_y(\theta_{f^{-1}(y)})),
$$
by mapping an element $u$ in spectrum degree $n$ to the element 
$$
(t_f(\theta)\wedge S^n)\circ u\circ d_f(\theta)\co \bigwedge_{y\in Y}S(\theta_{f^{-1}(y)})
\to \bigwedge_{y\in Y}T_y(\theta_{f^{-1}(y)})\wedge S^n.
$$
It is straight forward to check the required relations. We define $\Lambda'(T,M)$ to be the composite functor obtained by evaluating the homotopy colimit,   
$$
\Lambda'(T,M)\co \Fin_*\to\Sp^{\Sigma},\quad
\Lambda'_X(T,M)=\hocolim_{\I\langle X\rangle}F_X(T,M)\circ\pi_X.
$$
For a general based set $X$ let 
$$
\Lambda_X(T,M)=\hocolim_{Y\subseteq X}\Lambda'_Y(T,M)
$$
where the homotopy colimit is over the finite based subsets $Y$ in $X$. When $X$ is finite and $T$ and $M$ well-based we then have a chain of levelwise equivalences
$$
\Lambda_X(T,M)\xr{\sim}\Lambda'_X(T,M)\xr{\sim}\hocolim_{\I(X)}F_X(T,M).
$$ 
Finally, we extend this construction levelwise to a functor
$$
\Lambda(T,M)\co\Simp_*\to \Sp^{\Sigma}, 
$$
which as usual means that for a based simplicial set $X$, the $S$-module $\Lambda_X(T,M)$ is the realization of the simplicial $S$-module $[k]\mapsto \Lambda_{X_k}(T,M)$.  

In the case $T=M$ we also have an unbased version of this construction: write $\Lambda'(T)$ for the extension of $\Lambda'(T,T)$ to a functor on the category $\Fin$ of (unbased) finite sets such that there is a commutative diagram
$$
\begin{CD}
\Fin_*@>\Lambda'(T,T)>> \Sp^{\Sigma}\\
@VVV @|\\
\Fin@>\Lambda'(T)>>\Sp^{\Sigma}.
\end{CD}
$$ 
For a general set $X$ we define $\Lambda_X(T)$ to be the homotopy colimit
$$
\Lambda_X(T)=\hocolim_{Y\subseteq X}\Lambda_Y'(T)
$$
over the finite subsets $Y$ of $X$ and we again extend this levelwise to a functor
$
\Lambda(T)\co \Simp\to \Sp^{\Sigma}. 
$
Notice, that if $X$ has a specified base point, then there is a canonical levelwise equivalence
$\Lambda_X(T,T)\xr{\sim}\Lambda_X(T)$. Our consistent use of homotopy colimits implies that the homotopy invariance property of B\"okstedt's construction stated in Proposition 
\ref{Binvariant} carries over to the functors  
$\Lambda_X(T,M)$ and $\Lambda_X(T)$. 

\begin{proposition}\label{invarianceprop}
Applied to well-based commutative $S$-algebras $T$ and well-based $T$-modules $M$, the functors $\Lambda_X(T)$ and $\Lambda_X(T,M)$ take stable equivalences in $T$ and $M$ to stable equivalences of $S$-modules.  \qed
\end{proposition}

This result together with the next proposition shows that the functors $\Lambda_X(T,M)$ and 
$\Lambda_X(T)$ represent the derived Loday constructions on all based, respectively unbased, simplicial sets when $T$ and $M$ are well-based.
\begin{proposition}\label{LambdaLoday}
If $T$ is a cofibrant commutative $S$-algebra and $M$ a cofibrant $T$-module, then there are natural stable equivalences
$$
\Lambda_X(T,M)\simeq \mathcal L_X(T,M),\quad\text{and}\quad  \Lambda_X(T)\simeq \mathcal L_X(T). 
$$
\end{proposition}
\begin{proof}
We consider the case of $\Lambda_X(T,M)$; the case of $\Lambda_X(T)$ is similar. Let first $X$ be a finite based set. 
We then have the stable equivalence 
$$
\mathcal L_X(T,M)\simeq \hocolim_{\I}F_*(\mathcal L_X(T,M))
$$  
from (\ref{shipleyequi}) and it follows from the above discussion that if we choose an ordering of 
$X$, then there is a chain of stable equivalences  
$$
\Lambda_X(T,M)\xr{\sim}\Lambda'_X(T,M)\xr{\sim}\hocolim_{\I(X)}F_X(T,M)\xr{\sim}
\hocolim_{\I}F_*(\mathcal L_X(T,M)).
$$
However, since this chain of maps is not natural in $X$ it cannot be used directly to obtain a levelwise equivalence when applied to a simplicial set. For this reason we instead consider the map
$$
\Lambda_X'(T,M)=\hocolim_{\I\langle X\rangle}F_X(T,M)\circ \pi_X\to \hocolim_{\I}F_*(\mathcal L_X(T,M)),
$$
defined in analogy with (\ref{shipleyequi2}) but now induced by the functor obtained by applying 
$\I\langle-\rangle$ to the projection $X\to *$. The fact that (\ref{shipleyequi2}) is a stable equivalence easily implies the same for this map and we observe that in addition it is natural in $X$. In this way we obtain a chain of stable equivalences which are natural in the finite based set $X$, hence we get an induced chain of levelwise equivalences of simplicial $S$-modules when applied to a simplicial set. The cofibrancy assumptions on $T$ and $M$ imply that these simplicial objects are homotopically well-behaved so that the map of realizations is also a stable equivalence, see \cite{Sh}, Corollary 4.1.6.  
\end{proof}

\subsection{The multiplicative structure of $\Lambda_X(T,M)$}\label{multiplicativesection}
In this section $M$ denotes a commutative $T$-algebra. As recalled in Section \ref{LodayS}, the canonical map $T\to \mathcal L_X(T,M)$ is then a map of commutative $S$-algebras such that the Loday construction $\mathcal L_X(T,M)$ inherits a commutative $T$-algebra structure. For our homotopy invariant version it is easy to see that the analogous map $\Lambda_*(T)\to\Lambda_X(T,M)$ is a map of $S$-algebras but these $S$-algebras are not strictly commutative. (This is again due to the fact that the category $\I$ is not strictly symmetric). However, we show below that the map in question is a map of $E_{\infty}$ $S$-algebras such that we may view $\Lambda_*(T,M)$ as an $E_{\infty}$ 
$\Lambda_*(T)$-algebra. It is proved in \cite{EM} that any $E_{\infty}$ $S$-algebra may be rectified to a strictly commutative one, but for many purposes it is more convenient to work with the explicit $E_{\infty}$ structures described here. In the following theorem $\mathcal E$ denotes the Barratt-Eccles operad and we write $\Sp^{\Sigma}[\mathcal E]$ for the category of $S$-modules with $\mathcal E$-action; see Appendix \ref{hocolimEinfty} for details. Since 
$\mathcal E$ is an $E_{\infty}$ operad and the associativity operad (i.e., the operad with $n$th space equal to $\Sigma_n$) is augmented over $\mathcal E$, the objects in 
$\Sp^{\Sigma}[\mathcal E]$ are $E_{\infty}$ $S$-algebras by definition.  

\begin{proposition}\label{EinftyLambda}
If $M$ is a commutative $T$-algebra, then $\Lambda(T,M)$ lifts to a functor 
$$
\Lambda(T,M)\co \Simp_*\to\Sp^{\Sigma}[\mathcal E].
$$
\end{proposition}

We shall deduce this from the general machinery developed in Appendix \ref{appendix}. Consider first the case of a (not necessarily finite) based set $X$ and write $\mathcal F_X$ for the category of finite based subsets $Y$ in $X$ such that by definition $\Lambda_X(T,M)$ is the homotopy colimit of the $\mathcal F_X$-diagram $Y\mapsto \Lambda'_Y(T,M)$. Notice that the usual union of based subsets makes $\mathcal F_X$ a strict symmetric monoidal category. Given a family $Y_1,\dots,Y_n$ of finite based subsets of $X$ with union $Y$ we have the composite functor
\begin{equation}\label{thetasqcup}
\prod_{i=1}^n\I\langle Y_i\rangle\to \prod_{i=1}^n\I\langle Y\rangle\xr{\sqcup}\I\langle Y\rangle,
\quad (\theta_1,\dots,\theta_n)\mapsto \theta_1\sqcup\dots\sqcup\theta_n,
\end{equation}
where the first arrow is induced by the inclusions of $Y_i$ in $Y$ and the second arrow is given by the monoidal structure of $\I$. Thus, for a subset $V$ of $Y$,
$$
\theta_1\sqcup\dots\sqcup \theta_n(V)=\theta_1(Y_1\cap V)\sqcup\dots
\sqcup\theta_n(Y_n\cap V). 
$$
Letting the $Y_i$ vary we may view this as a natural transformation of 
$\mathcal F_X^n$-diagrams. We shall also need the associated map of $S$-algebras
$$
\bigwedge_{i=1}^n\bigwedge_{y\in Y_i}T_y(\theta_i\{y\})\xr{\sim}\bigwedge_{y\in Y}\bigwedge_{\{i: y\in Y_i\}}
T_y(\theta_i\{y\})\to \bigwedge_{y\in Y}T_y(\theta_1\sqcup\dots\sqcup\theta_n\{y\}).
$$
The first map is the obvious permutation of smash factors and the second is induced by the multiplications in $T$ and $M$. Proceeding as in the definition of the map 
(\ref{phitransformation}) we get from this a natural transformation of $\prod_{i}\I\langle Y_i\rangle$-diagrams
\begin{equation}\label{Einftymultiplication}
\bigwedge_{i=1}^nF_{Y_i}(T,M)\circ \pi_{Y_i}(\theta_i)\to F_Y(T,M)\circ
\pi_Y(\theta_1\sqcup\dots\sqcup\theta_n).
\end{equation}

\bigskip
\noindent\textit{Proof of Proposition \ref{EinftyLambda}.}
As above we first consider a based set $X$ and claim that $\Lambda_X(T,M)$ has a canonical $\mathcal E$-action. For this it suffices by Corollary \ref{BEhocolim} that $\mathcal E$ acts on the $\mathcal F_X$-diagram $\Lambda'(T,M)$, thought of as an object in the symmetric monoidal category $\mathcal F_X\Sp^{\Sigma}$ of functors from $\mathcal F_X$ to $\Sp^{\Sigma}$. As explained in Appendix \ref{hocolimEinfty} we must therefore exhibit a sequence of natural transformations 
$$
\theta_n\co\mathcal E(n)_+\wedge\Lambda'_{Y_1}(T,M)\wedge\dots\wedge\Lambda'_{Y_n}(T,M)\to\Lambda'_{Y}(T,M)
$$
of $\mathcal F_X^n$-diagrams, satisfying the standard unitality, associativity and equivariance conditions required for an operad action. Using the simplicial description of 
$\mathcal E(n)$ given in (\ref{BEsimplicial}) and the fact that topological realization preserves smash products it suffices to define a sequence of maps of simplicial $S$-modules
$$
\Sigma^{k+1}_{n+}\wedge\bigwedge_{i=1}^n\bigvee_{\{\theta_0^i\leftarrow\dots\leftarrow \theta_k^i\}}
F_{Y_i}(T,M)\circ\pi_{Y_i}(\theta_k^i)\to
\bigvee_{\left\{\psi_0\leftarrow\dots\leftarrow\psi_k\right\}} F_Y(T,M)\circ\pi_Y(\psi_k)
$$    
satisfying the analogous coherence conditions. For this we observe that (\ref{thetasqcup}) extends over the translation category $\tilde\Sigma_n$ in the sense that there is a functor
$$
\tilde\Sigma_n\times\I\langle Y_1\rangle\times\dots\I\langle Y_n\rangle\to \I\langle Y\rangle,\quad (\sigma,\theta_1,
\dots\theta_n)\mapsto \theta_{\sigma^{-1}(1)}\sqcup\dots\sqcup\theta_{\sigma^{-1}(n)},
$$ 
defined in analogy with (\ref{BEaction}). Keeping an element $(\sigma_0,\dots,\sigma_k)$ in 
$\Sigma_n^{k+1}$ fixed we then use this functor together with the map in 
(\ref{Einftymultiplication}) to map the wedge component indexed by $\theta_0^i\leftarrow\dots\leftarrow\theta_k^i$ for $i=1,\dots,n$, to the wedge component in the target indexed by 
$$
\theta_0^{\sigma_0^{-1}(1)}\sqcup\dots\sqcup\theta_0^{\sigma_0^{-1}(n)}\leftarrow\dots\leftarrow
\theta_k^{\sigma_k^{-1}(1)}\sqcup\dots\sqcup\theta_k^{\sigma_k^{-1}(n)}.
$$
In detail, the map in question is given by the composition 
\begin{align*}
\bigwedge_{i=1}^nF_{Y_i}(T,M)\circ\pi_{Y_i}(\theta_k^i)&\to 
F_Y(T,M)\circ\pi_Y(\theta_k^1\sqcup\dots\sqcup\theta_k^n)\\
&\to F_Y(T,M)\circ\pi_Y(\theta_k^{\sigma^{-1}(1)}\sqcup\dots\sqcup \theta_k^{\sigma^{-1}(n)})
\end{align*}
where the last map is induced by the permutation $\sigma_k$. 
This should be compared with the definition of the action in the proof of Proposition 
\ref{hocolimEinftyprop}.  
We next observe that a map of based sets $f\co X\to X'$ induces a strict symmetric monoidal functor $f_*\co\mathcal F_X\to\mathcal F_{X'}$ such that the associated natural transformation of $\mathcal F_X$-diagrams
$$
\Lambda_Y'(T,M)\to\Lambda_{f(Y)}'(T,M),\quad Y\subseteq X,
$$
is an $\mathcal E$-map in $\mathcal F_X\Sp^{\Sigma}$. Hence the induced map 
$\Lambda_X(T,M)\to\Lambda_{X'}(T,M)$ of homotopy colimits is a morphism in $\Sp^{\Sigma}[\mathcal E]$ by Corollary \ref{BEhocolim}. Let finally $X$ be a based simplicial set. Then it follows from the above that the correspondence $[k]\mapsto\Lambda_{X_k}(T,M)$ defines a simplicial object in 
$\Sp^{\Sigma}[\mathcal E]$ and since operad actions are preserved under topological realization we conclude that $\Lambda_X(T,M)$ inherits an $\mathcal E$-action which is clearly natural in 
$X$. \qed

\medskip
By an analogous argument we also have the following unbased version
\begin{proposition}
For a commutative $S$-algebra $T$, the functor $\Lambda(T)$ lifts to a functor
$$
\Lambda(T)\co\Simp\to\Sp^{\Sigma}[\mathcal E].
\eqno\qed
$$
\end{proposition}

\begin{remark}
Let again $M$ be a commutative $T$-algebra. Imitating the construction of Hesselholt-Madsen from \cite{HM}, Proposition 2.7.1, one can modify the definition of 
$\Lambda_X(T,M)$ to give a strictly commutative $S$-algebra when $T$ and $M$ are 
convergent. The main point is to redefine the $n$th space of $\Lambda'_X(T,M)$ to be a certain homotopy colimit over the $n$-fold product category $\I\langle X\rangle^n$.  
For convergent $T$ and $M$ it then follows from 
B\"okstedt's approximation lemma, \cite{M}, Lemma 2.3.7, that the resulting commutative $S$-algebra is stably equivalently to $\Lambda_X(T,M)$. We shall not make use of this construction and leave the details with the reader.
\end{remark}

\section{$\I$-spaces and symmetric Thom spectra}\label{Isymmetricthomsection}
In this section we collect the necessary background material on $\I$-spaces and symmetric Thom spectra. 

\subsection{The homotopy theory of $\I$-spaces}
We write $\I\mathcal U$ for the category of $\I$-spaces, that is, functors $A\co\I\to\mathcal U$. As is generally the case for a category of diagram spaces defined on a small symmetric monoidal category  
this inherits a symmetric monoidal structure, here written $\boxtimes$, such that for two $\I$-spaces $A$ and $B$,
$$
A\boxtimes B(n)=\colim_{\mathbf n_1\sqcup\mathbf n_2\to \mathbf n}A(n_1)\times B(n_2)
$$
where the colimit is over the category $\sqcup/\mathbf n$ associated to the concatenation functor $\sqcup\co\I^2\to \I$. The constant $\I$-diagram $\I(\mathbf 0,-)$ is the unit for the monoidal structure. We define an \emph{$\I$-space monoid} to be a monoid in 
$\I\mathcal U$ under this monoidal product. There is an obvious inclusion of 
$\mathcal U$ in $\I\mathcal U$ as the subcategory of constant $\I$-spaces.  The colimit functor  provides a retraction of this but we shall be more interested in the homotopy colimit functor from $\I\mathcal U$ to $\mathcal U$.
We say that a map of $\I$-spaces is an \emph{$\I$-equivalence} if the induced map of homotopy colimits is a weak homotopy equivalence of spaces. This is the analogue for $\I$-spaces of a stable equivalence of $S$-modules. It is proved in \cite{Sch3} that the $\I$-equivalences are the weak equivalences in a model structure on $\I\mathcal U$ whose fibrations, which we refer to as $\I$-fibrations, are the levelwise fibrations $A\to B$ such that the diagrams
\begin{equation}\label{Ifibrations}
\begin{CD}
A(m)@>>>B(m)\\
@VVV @VVV\\
A(n)@>>> B(n)
\end{CD}
\end{equation}
are homotopy cartesian for all morphisms $\mathbf m\to\mathbf n$ in $\I$. (The cofibrations also have an explicit description but we shall not need this here). With this model structure the inclusion and colimit functors define a Quillen equivalence between $\I\mathcal U$ and 
$\mathcal U$ with its usual model structure. There also is a \emph{positive model structure} on 
$\I\mathcal U$ whose weak equivalences are again the $\I$-equivalences and whose fibrations are the positive $\I$-fibrations. 
This is the analogue of the positive model structure on $S$-modules introduced in  
\cite{MMSS}. 
We say that an $\I$-space $A$ is 
\emph{convergent} if there exists an unbounded, non-decreasing sequence of integers
$\lambda_m$ such that a morphism $\mathbf m\to\mathbf n$ in $\I$ induces a $\lambda_m$-connected map $A(m)\to A(n)$. The following observation will be needed later.

\begin{lemma}\label{Ifibrationconvergence}
Let $A\to B$ be an $\I$-fibration or a position $\I$-fibration and suppose that $B$ is convergent. Then $A$ is also convergent.\qed
\end{lemma}    
As for symmetric spectra the positive model structure on $\I\mathcal U$ induces a model structure on the category of commutative monoids.

\begin{proposition}[\cite{Sch3}]\label{commutativemodel}
There is a model structure on the category of commutative $\I$-space monoids whose weak equivalences and fibrations are the $\I$-equivalences and positive $\I$-fibrations of the underlying $\I$-spaces.\qed    
\end{proposition}

Let $\mathcal C$ be an operad in the sense of \cite{May} and write $C$ for the associated monad on $\I\mathcal U$, that is, 
$$
C(A)=\coprod_{k\geq 0}\mathcal C(k)\times_{\Sigma_k}A^{\boxtimes k}.
$$
We say that $\mathcal C$ acts on an $\I$-space $A$ if $A$ is an algebra for this monad and we write $\I\mathcal U[\mathcal C]$ for the category of $\I$-spaces with 
$\mathcal C$-action. By the universal property of the $\boxtimes$-product, a $\mathcal C$-action amounts to a sequence of natural transformations 
$$
\mathcal C(k)\times A(n_1)\times\dots\times A(n_k)\to A(n_1+\dots+n_k)  
$$
where for each $k$ we view both sides as functors from $\I^k$ to $\mathcal U$. These maps are required to satisfy the usual unitality, associativity and equivariance conditions, see \cite{May}, Lemma 1.4. Since the inclusion of $\mathcal U$ in $\I\mathcal U$ is strong symmetric monoidal, a $\mathcal C$-action on a space induces a $\mathcal C$-action on the associated constant $\I$-space. One of the advantages of working in the category $\I\mathcal U$ is that $E_{\infty}$ objects can be rectified to strictly commutative objects in the sense of the following theorem. We formulate this only for the Barratt-Eccles operad $\mathcal E$ but the result holds in general for $E_{\infty}$ operads under mild point-set topological conditions. Notice, that since the associativity operad is augmented over $\mathcal E$, an $\mathcal E$-$\I$-space is automatically an $\I$-space monoid.

\begin{theorem}[\cite{Sch3}]\label{Irectification}
Let $A$ be an $\I$-space with an action of the Barratt-Eccles operad $\mathcal E$. There is a functorial chain of $\I$-equivalences of 
$\mathcal E$-$\mathcal I$-space monoids 
$$
A\xl{\alpha} \check A\xr{\beta}\hat A
$$ 
where $\hat A$ is commutative. If $A$ is already commutative, then there is a natural $\I$-equivalence of commutative $\I$-space monoids $\gamma\co\hat A\to A$ such that $\alpha=\gamma\beta$.
\end{theorem}

In fact there are different options for such a rectification process: one may apply a method similar to that used for symmetric spectra in \cite{EM} or one may use a bar construction type argument 
as in \cite{May}. 
 
\subsection{Commutative $\I$-space monoids and $E_{\infty}$ spaces}\label{Iinfinite}
There are several good ways in which one can associate to a commutative $\I$-space monoid 
$A$ an infinite loop space. For our purposes it is convenient to apply the construction from 
\cite{Sch1}, Section 5.2, that to $A$ associates a special $\Gamma$-space $A_{h\I}$ whose underlying space $A_{h\I}(S^0)$ is the (unbased) homotopy colimit of $A$ over $\I$.  Here we use the term \emph{special $\Gamma$-space} in the sense of \cite{BF}, that is, $A_{h\I}$ is a 
$\Fin_*$-diagram such that $A_{h\I}(*)$ is a one-point space and the canonical map from 
$A_{h\I}(X\vee Y)$ to $A_{h\I}(X)\times A_{h\I}(Y)$ is a weak homotopy equivalence. We briefly recall the definition of $A_{h\I}$, referring to \cite{Sch1} for details. The first observation is that the $\Fin$-diagram $\I\langle-\rangle$ from Section \ref{invariantsection} extends to a 
$\Fin_*$-diagram $\I_*\langle-\rangle$ in the sense that there is a commutative diagram of functors
$$
\xymatrix{
\Fin\ar[rr]^{(-)_+}\ar[dr]_{\I\langle-\rangle} & & \Fin_* \ar[dl]^{\I_*\langle-\rangle}\\
& \Cat &
}
$$
where $(-)_+$ is the functor that to an unbased set adjoins a disjoint base point. 
Indeed, given a based set $X$ one defines $\I_*\langle X\rangle$ to be the category $\I\langle \bar X\rangle$ where $\bar X$ denotes the unbased set obtained by removing the base point from 
$X$ (this is the category denoted $\mathcal D(X)$ in \cite{Sch1}). Given a commutative 
$\I$-space monoid $A$, let $A\langle X\rangle$  be the $\I_*\langle X\rangle$-diagram defined by the composition 
$$
A\langle X\rangle\co \I_*\langle X\rangle\xr{\pi_{\bar X}}\I(\bar X)
\xr{\prod_{x\in \bar X}A}\mathcal U,\quad \theta\mapsto \textstyle\prod_{x\in \bar X}A(\theta_x),
$$
and write $A'_{h\I}(X)$ for the unbased homotopy colimit
$$
A_{h\I}'(X)=\hocolim_{\I\langle X\rangle}A\langle X\rangle.
$$  
For a (not necessarily finite) based set $X$ we then define $A_{h\I}(X)$ to be the homotopy colimit 
$$
A_{h\I}(X)=\hocolim_{Y\subseteq X}A'_{h\I}(Y)
$$
over the finite based subsets $Y$ in $X$ and we finally extend this levelwise to a functor
$
A_{h\I}\co\Simp_*\to \mathcal U.
$
It is proved in \cite{Sch1}, Proposition 5.3, that this construction defines a special $\Gamma$-space and that consequently $\Omega A_{h\I}(S^1)$ is a group completion of the topological monoid $A_{h\I}(S^0)$. We say that the commutative $\I$-space monoid $A$ is \emph{grouplike} if $A_{h\I}(S^0)$ is a grouplike topological monoid. By definition \cite{BF} this is equivalent to the $\Gamma$-space $A_{h\I}$ being \emph{very special}.   
Let us write $\mathbf A$ for the associated $S$-module with $n$th space 
$A_{h\I}(S^n)$. This is a positive $\Omega$-spectrum in the sense that the adjoint structure maps are weak homotopy equivalences except possibly in degree 0. Let 
$\mathcal U[\mathcal E]$ be the category of spaces equipped with an action of the Barratt-Eccles operad. 

\begin{proposition} \label{GammaEinfty}
The $\Simp_*$-diagram $A_{h\I}$ lifts to a functor
$A_{h\I}\co\Simp_*\to \mathcal U[\mathcal E]$.
\end{proposition}
 
\begin{proof}
We shall again deduce this from the general machinery in Appendix \ref{appendix}.
The definition of the $\mathcal E$-action is similar to that used in the proof of Proposition 
\ref{EinftyLambda}. Thus, given a based set $X$ we again let $\mathcal F_X$ be the category of finite based subsets $Y$ in $X$ and view $A_{h\I}'$ as an $\mathcal F_X$-diagram, 
$Y\mapsto A_{h\I}'(Y)$. Using Corollary \ref{spaceEinfty} we must then show that $\mathcal E$ acts on $A_{h\I}'$ as an object in the symmetric monoidal category of functors from $\mathcal F_X$ to 
$\mathcal U$. As in the $S$-module analogue  such an action amounts to a compatible sequence of natural transformations of $\mathcal F_X^n$-diagrams
$$
\mathcal E(n)\times A_{h\I}'(Y_1)\times \dots\times A_{h\I}'(Y_n)\to
A_{h\I}'(Y).
$$  
Here $Y_1,\dots,Y_n$ denotes a family of finite based subsets of $X$ with union $Y$. For this we observe that there is an $\I_*\langle-\rangle$-analogue of (\ref{thetasqcup}) and an associated natural transformation of $\prod_{i=1}^n\I_*\langle Y_i\rangle$-diagrams
$$
\prod_{i=1}^nA\langle Y_i\rangle(\theta_i)\cong \prod_{y\in Y}\prod_{\{i:y\in Y_i\}}A(\theta_i\{y\})
\to A\langle Y\rangle(\theta_1\sqcup\dots\sqcup \theta_n).
$$ 
Using these structure maps the construction of the action proceeds exactly as in the proof of Proposition \ref{EinftyLambda}.
\end{proof} 
 
 It follows from this that a commutative $\I$-space monoid $A$ gives rise to an $S$-module with 0th space $A_{h\I}(S^0)$ in two different ways: on the one hand we have the $S$-module 
 $\mathbf A$ arising from the $\Gamma$-space structure and on the other hand we have the $S$-module arising from the $E_{\infty}$ structure. We want to show that these $S$-modules are in fact equivalent.  Recall from \cite{May} and \cite{MM}, Chapter I.8,  that May has defined a functor $\mathcal U[\mathcal E]\to\Sp^{\Sigma}$ which to an $E_{\infty}$ space $Z$ 
associates an $S$-module $\mathbf BZ$ whose 0th space equals $Z$ and whose $n$th space $B_nZ$ is a certain bar construction on $Z$ defined using the operad structure. Applying this construction to $A_{h\I}$ we obtain a symmetric bispectrum $\mathbf B\mathbf A$ with 
$(m,n)$th space $B_mA_{h\I}(S^n)$. Since $A_{h\I}$ is special it follows that this is a positive 
$\Omega$-bispectrum in the sense that the adjoint structure maps are weak equivalences except in bidegree $(0,0)$. This in turn implies that the chain of maps
\begin{align*}
A_{h\I}(S^n)&\to\hocolim_{m_1,m_2}\Omega^{m_1+m_2}(B_{m_1}A_{h\I}(S^{m_2+n}))\\
&\leftarrow 
\hocolim_{m_1,m_2}\Omega^{m_1+m_2}(B_{m_1}A_{h\I}(S^{m_2})\wedge S^n)\\
&\to
\hocolim_{m_1,m_2}\Omega^{m_1+m_2}(B_{m_1+n}A_{h\I}(S^{m_2}))
\leftarrow B_nA_{h\I}(S^0)
\end{align*}
 is a chain of weak equivalences for $n$ positive. Viewing each of these maps as a map of $S$-modules we get the required stable equivalence between the two $S$-modules associated to 
 $A$.

\subsection{The symmetric Thom spectrum functor}\label{symmetricthomsection}
In this section we briefly recall the basic properties of the symmetric Thom spectrum functor. The main reference for this material is the paper \cite{Sch2} which in turn is based in part on the foundational work by Lewis \cite{LMS}. Let $F(n)$ be the topological monoid of based homotopy equivalences of $S^n$ and let $BF(n)$ be its classifying space. We write 
$V(n)$ for the one-sided bar construction $B(*,F(n),S^n)$ and $p_n\co V(n)\to BF(n)$ for the canonical quasifibration. Following Lewis \cite{LMS}, Section IX, the Thom space functor assigns to a map $f\co X\to BF(n)$ the based space $T(f)=f^*V(n)/X$
obtained by first pulling $V(n)$ back via $f$ and then forming the quotient by $X$. Defined in this way the Thom space construction is a functor from the category $\mathcal U/BF(n)$ of spaces over $BF(n)$ to the category of based spaces. Since quasifibrations are not in general preserved under pullbacks this is not a homotopy functor on the full category $\mathcal U/BF(n)$. In order to identify the "correct" homotopy type of the Thom space associated to an object $f$ as above we again follow Lewis and consider the endofunctor $\Gamma$ on $\mathcal U/BF(n)$ that to $f$ associates the usual Hurewicz fibration $\Gamma(f)\co\Gamma_f(X)\to BF(n)$, where 
$$
\Gamma_f(X)=\{(x,\omega)\in X\times BF(n)^{\I}\co f(x)=\omega(0)\}
$$
and $\Gamma(f)$ maps a point $(x,\omega)$ to $\omega(1)$. It is proved in \cite{LMS} that the composite functor $T\circ\Gamma$ is a homotopy functor and we say that an object 
$f$ in 
$\mathcal U/BF(n)$ is \emph{$T$-good} if the canonical equivalence $X\to \Gamma_f(X)$ induces a weak homotopy equivalence  of Thom spaces. Thus, the Thom space functor is a homotopy functor on the full subcategory of $T$-good objects in $\mathcal U/BF(n)$. The objects that occur naturally in the applications are usually $T$-good; in particular this is the case for an object that is a Hurewicz fibration and for any object factoring through the classifying space $B\mathit{Top}(n)$ for the group of base point preserving homeomorphisms of $S^n$. 

We now turn to symmetric Thom spectra. It is well-known that the correspondence 
$\mathbf n\mapsto BF(n)$ defines a commutative $\I$-space monoid with monoidal structure  induced by the monoid homomorphisms $F(m)\times F(n)\to F(m+n)$ that map a pair of equivalences to their smash product, see e.g.\ \cite{Sch2}. We write $\I\mathcal U/BF$ for the category of $\I$-spaces over $BF$ with its induced symmetric monoidal structure. 
The symmetric Thom spectrum functor 
\begin{equation}\label{symthom}
T\co \I\mathcal U/BF\to \Sp^{\Sigma}
\end{equation}
is defined by levelwise application of the Thom space functor: given a map of $\I$-spaces $f\co A\to BF$, the Thom $S$-module $T(f)$ has as its $n$th space the Thom space $T(f_n)$. It is proved in \cite{Sch2} that $T$ is a strong symmetric monoidal functor, hence in particular that it takes commutative monoids (that is, maps of commutative 
$\I$-space monoids $A\to BF$) to commutative $S$-algebras. We define an endofunctor 
$\Gamma$ on $\I\mathcal U/BF$ by applying the fibrant replacement functor considered above at each level and we say that an object $f$ is \emph{$T$-good} if the canonical map of $S$-modules $T(f)\to T(\Gamma(f))$ is a stable equivalence. In particular, $\Gamma(f)$ is always 
$T$-good and it follows from \cite{LMS}, Proposition 1.11, that the $S$-module $T(\Gamma(f))$ is well-based. 
Notice also that $\Gamma$ induces an endo\-functor on the category of commutative monoids in $\I\mathcal U/BF$. 
 The following theorem from \cite{Sch2} implies that the symmetric Thom spectrum functor is a homotopy functor on the full subcategory of $T$-good objects in $\I\mathcal U/BF$.

\begin{theorem}[\cite{Sch2}]\label{Thominvariance}
The symmetric Thom spectrum functor takes $\I$-equivalences between $T$-good objects in 
$\I\mathcal U/BF$ to stable equivalences of $S$-modules.\qed
\end{theorem}

We also want to associate Thom $S$-modules to space level data and for this purpose it is convenient to use the homotopy colimit $BF_{h\I}$ of $BF$ as a model for the classifying space for spherical fibrations (thus, in the notation used above we write $BF_{h\I}$ for the space $BF_{h\I}(S^0)$). It is proved in \cite{Sch2} that the homotopy colimit functor from $\I\mathcal U/BF$ to $\mathcal U/BF_{h\I}$ has a homotopy inverse
$$
R\co \mathcal U/BF_{h\I}\to \I\mathcal U/BF, \quad (B\xr{f}BF_{h\I})\mapsto (R_f(B)\xr{R(f)}BF)
$$
with values in the subcategory of $T$-good objects. 
Composing with the Thom spectrum functor (\ref{symthom}) therefore gives a homotopically well-behaved symmetric Thom spectrum functor on 
$\mathcal U/BF_{h\I}$. Furthermore, $BF_{h\I}$ has a canonical action of the Barratt-Eccles operad and it follows from \cite{Sch2} that the functor $R$ 
takes $E_{\infty}$ spaces over $BF_{h\I}$ to $E_{\infty}$ $\I$-spaces over $BF$.
There are now two obvious ways to assign a commutative Thom $S$-algebra to an 
$E_{\infty}$ map $B\to BF_{h\I}$. On the one hand we could apply the symmetric Thom spectrum functor on $\mathcal U/BF_{h\I}$ to get an $E_{\infty}$ $S$-algebra and then use the rectification procedure \cite{EM} on the level of $E_{\infty}$ $S$-algebras to turn  this into a strictly commutative $S$-algebra. On the other hand we could use the rectification procedure in Theorem \ref{Irectification} to the $E_{\infty}$ $\I$-space monoid $R_f(B)$ and then  
apply the symmetric Thom spectrum functor (\ref{symthom}) to the associated commutative 
$\I$-space monoid $\hat R_f(B)$ over $BF$. It is the second construction that is most convenient for our purpose so we turn it into a definition.

\begin{definition}\label{commThom}
Let $f\co B\to BF_{h\I}$ be a map of $E_{\infty}$ spaces. The commutative Thom $S$-algebra $T(f)$ is defined by applying the symmetric Thom spectrum functor (\ref{symthom}) to the map of commutative $\I$-space monoids $\hat R_f(B)\to \hat BF\xr{\gamma} BF$.
\end{definition} 

A comparison of the symmetric Thom spectrum described here to the Lewis-May Thom spectrum functor from \cite{LMS} can be found in \cite{Sch2}.

\section{Higher topological Hochschild homology of Thom spectra}\label{highthomsection}
It is illuminating to first sketch a proof of Theorem \ref{Einftyhighthom} ignoring all technical details. Thus, let $f\co B\to BF_{h\I}$ be a map of $E_{\infty}$ spaces with $B$ grouplike, and let us use the same notation $f\co A\to BF$ for the associated map of commutative $\I$-space monoids defined as in Definition \ref{commThom}. Applying the $\I$-space analogue of the Loday functor to the commutative $\I$-space monoid $A$ we get a commutative diagram
$$
\xymatrix{
\mathcal L_X(A,A)\ar[rr] \ar[dr] & & \mathcal L_*(A)\times \mathcal L_X(A,*) \ar[dl]\\
& BF &
}
$$
in $\I\mathcal U$. The first component of the upper horizontal map is induced by the projection $X\to *$ and the second component is induced by the projection $A\to *$ onto the terminal $\I$-space. Identifying $\mathcal L_*(A)$ with $A$, the map on the right is defined by first projecting onto the first factor and then applying $f$. One can check that the upper horizontal map is an $\I$-equivalence if $A$ is cofibrant and since the symmetric Thom spectrum functor is strong symmetric monoidal and commutes with topological realization (see \cite{Sch2}) we can identify the Thom spectrum $T(\mathcal L_X(A,A))$ with the Loday construction $\mathcal L_X(T(f),T(f))$.
Using Theorem \ref{Thominvariance} we would therefore expect to get a chain of stable equivalences of commutative $S$-algebras
$$
\mathcal L_X(T(f),T(f))\simeq T(\mathcal L_X(A,A))\simeq T(\mathcal L_*(A)\times \mathcal L_X(A,*))\simeq T(f)\wedge \mathcal L_X(A,*)_+
$$
and it is believable that the last term can be identified with 
$T(f)\wedge\Omega^{\infty}(\mathbf A\wedge X)_+$. Given a $T(f)$-module $M$, the result should then follow by applying the functor 
$M\wedge_{T(f)}(-)$. However, when trying to make this argument precise one encounters several technical difficulties. First of all one needs $A$ and $T(f)$ to be cofibrant in order to control the homotopy type of the Loday functors but unfortunately the Thom $S$-algebra associated to a cofibrant $\I$-space $A$ need not be cofibrant. It is also not clear that the $T$-goodness condition for an object in $\mathcal \I\mathcal U/BF$ is preserved under cofibrant replacement. (The latter difficulty is caused by the technical subtlety that whereas Hurewicz cofibrations are preserved under pullback along Hurewicz fibrations, the behavior under pullback along Serre fibrations is not well understood). It is for these reasons, along with the equivariant applications in future work, that we prefer to work with the homotopy invariant model $\Lambda_X(T(f),M)$. Using this model the argument involves no point-set topological considerations at the price of being slightly more involved combinatorially. 
We shall deduce Theorem \ref{Einftyhighthom} from the following $\I$-space analogue. 
Recall that a commutative $\I$-space monoid $A$ is grouplike if the homotopy colimit 
$A_{h\I}(S^0)$ is a grouplike topological monoid. 
\begin{theorem}\label{Ispacehighthom}
Let $f\co A\to BF$ be a $T$-good map of commutative $\I$-space monoids and let $T(f)$ be the associated commutative Thom $S$-algebra. If $A$ is grouplike, then there is a chain of stable equivalences of $S$-modules
$$
\Lambda_X(T(f),M)\simeq M\wedge A_{h\I}(X)_+
$$
for any based simplicial set $X$ and any well-based $T(f)$-module $M$. If $M$ is a commutative $T(f)$-algebra, then this is an equivalence of $E_{\infty}$ $S$-algebras.  
\end{theorem}
For the last statement we recall from Proposition \ref{GammaEinfty} that the $\Simp_*$-diagram $A_{h\I}$ takes values in the category of $E_{\infty}$ spaces. 

\medskip
\noindent\textit{Proof of Theorem \ref{Einftyhighthom}.}
Starting with a map of $E_{\infty}$ spaces $f\co B\to BF_{h\I}$ we proceed as in Definition 
\ref{commThom} and write $f\co A\to BF$ for the associated map of commutative $\I$-space monoids. Thus, in the notation from Theorem \ref{Irectification}, $A$ equals the commutative 
$\I$-space monoid $\hat{R}_g(B)$ and it follows from Corollary \ref{spaceEinfty} that we have a chain of weak homotopy equivalences of $E_{\infty}$ spaces
$$
\hocolim_{\I}A\simeq\hocolim_{\I}R_f(B)\simeq B.
$$
The last equivalence is derived from the fact that the functor $R$ is a homotopy inverse of the homotopy colimit functor, see \cite{Sch2}. It follows that if the $E_{\infty}$ space $B$ is grouplike, then so is the commutative $\I$-space monoid $A$. We conclude from the discussion at the end of Section \ref{Iinfinite}, that the $S$-module arising from the operad action on $B$ is stably equivalent to the $S$-module $\mathbf A$ obtained from the special $\Gamma$-space 
$A_{h\I}$. By definition, the commutative $S$-algebra $T(f)$ in Theorem 
\ref{Einftyhighthom} is the symmetric Thom spectrum functor (\ref{symthom}) applied to the map $f$ above. It is implicit in the statement of Theorem \ref{Einftyhighthom} that when applying the Loday functor we should first choose a flat replacement $T(f)^c$ of $T(f)$ and a flat replacement $M^c$ of $M$ as a $T(f)^c$-module. It then follows from Propositions \ref{invarianceprop} and \ref{LambdaLoday} that there are stable equivalences
$$
\mathcal L_X(T(f)^c,M^c)\simeq \Lambda_X(T(f)^c,M^c)\simeq \Lambda_X(T(f),M)
$$  
which are equivalences of $E_{\infty}$ $S$-algebras if $M$ is a commutative $T$-algebra. (As a technical note we remark that we may assume without loss of generality that $T(f)$ is well-based by applying the fibrant replacement functor $\Gamma$ from Section \ref{symmetricthomsection} to $f$ if necessary). Using Theorem \ref{Ispacehighthom} it therefore remains to identify $A_{h\I}(X)$ with $\Omega^{\infty}(\mathbf A\wedge X)$. Let us write $\mathbf A(S\wedge X)$ for the 
$S$-module with $n$th space $A(S^n\wedge X)$. As remarked above, the assumption that $B$ be grouplike implies that $A_{h\I}(S^0)$ is a grouplike topological monoid. Hence the $\Gamma$-space $A_{h\I}$ is very special and consequently $\mathbf A(S\wedge X)$ is an $\Omega$-spectrum, see \cite{BF}, Theorem 4.2. The conclusion now follows from   
\cite{BF}, Theorem 4.1, which states that the canonical map of $S$-modules 
$\mathbf A\wedge X\to \mathbf A(S\wedge X)$ is a stable equivalence.\qed

\medskip
We now describe the chain of maps in Theorem \ref{Ispacehighthom} and write $T$ for the commutative $S$-algebra $T(f)$. By naturality of the constructions it suffices to define a chain of natural transformations of $\Fin_*$-diagrams. For this purpose we introduce the auxiliary functor 
$$
F_*(A\langle X\rangle,M)\co \I\times \I_*\langle X\rangle\to\Sp^{\Sigma}
$$
which to a finite based set $X$ and a pair of objects $j$ in $\I$ and $\theta$ in 
$\I_*\langle X\rangle$ associates the $S$-module
$$
F_*(A\langle X\rangle,M)(j,\theta)=F_S(S(j),M(j)\wedge A\langle X\rangle(\theta)_+).
$$
In this way we have defined a functor $\Fin_*\to\mathcal D(\Sp^{\Sigma})$ and we write 
$\Lambda'(A,M)$ for the composite functor obtained by evaluating the homotopy colimit,
$$
\Lambda'(A,M)\co \Fin_*\to \Sp^{\Sigma},\quad \Lambda_X'(A,M)=
\hocolim_{\I\times \I_*\langle X\rangle}F_*(A\langle X\rangle,M).
$$
Consider the natural transformation of $\Fin_*$-diagrams
$$
p_X\times q_X\co \I\langle X\rangle\to\I\times \I_*\langle X\rangle,
$$
where $p_X$ is induced by the projection $X\to *$ (identifying $\I$ with $\I\langle *\rangle$) and $q_X$ is defined by first identifying $\I\langle X\rangle$ with $\I_*\langle X_+\rangle$ and then applying the based map $X_+\to X$ that is the identity on $X$. We shall extend this to a natural transformation of $\Fin_*$-diagrams in $\mathcal D(\Sp^{\Sigma})$,
$$
(\I\langle X\rangle, F_X(T,M)\circ\pi_X)\to (\I\times\I_*\langle X\rangle,F_*(A\langle X\rangle,M)),
$$
hence we need a natural transformation 
\begin{equation}\label{pqtransformation}
F_X(T,M)\circ\pi_X\to F_*(A\langle X\rangle,M)\circ p_X\times q_X
\end{equation}
of $\I\langle X\rangle$-diagrams for each finite based set $X$. For this purpose we first fix an object $\theta$ in $\I\langle X\rangle$ and observe that there is a commutative diagram
$$
\begin{CD}
\displaystyle\prod_{x\in \bar X}A(\theta_x)@>>>A(\coprod_{x\in \bar X}\theta_x)\times 
\displaystyle\prod_{x\in\bar X}A(\theta_x)\\
@VVV @VVV\\
BF(\coprod_{x\in \bar X}\theta_x)@=BF(\coprod_{x\in \bar X}\theta_x)
\end{CD}
$$
where the first component of the upper map is the multiplication in $A$ and the second component is the identity. The vertical map on the right is defined by projecting onto the first factor and then applying $f$. Evaluating the Thom spaces of the vertical maps and using that 
the Thom space functor takes product to smash products, we get from this a map of based spaces
$$
\bigwedge_{x\in \bar X}T(\theta_x)\to T({\textstyle\coprod_{x\in \bar X}\theta_x})\wedge\ A\langle X\rangle(\theta)_+
$$ 
which in turn induces a map
\begin{equation}
\begin{aligned}\label{Mthommap}
M(\theta_*)\wedge\bigwedge_{x\in \bar X}T(\theta_x)
&\to M(\theta_*)\wedge 
T(\textstyle{\coprod}_{x\in \bar X}\theta_x)\wedge A\langle X\rangle(\theta)_+ \\
&\to
M(\textstyle{\coprod}_{x\in X}\theta_x)\wedge A\langle X\rangle(\theta)_+\to M(\theta_X)\wedge 
A\langle X\rangle(\theta)_+
\end{aligned}
\end{equation}
using the $T$-module structure on $M$. Letting $\theta$ vary, the induced map
$$
F_S(\bigwedge_{x\in X}S(\theta_x),\bigwedge_{x\in X}T_x(\theta_x))\to
F_S(S(\theta_X),M(\theta_X)\wedge A\langle X\rangle(\theta)_+)
$$
gives rise to the natural transformation in (\ref{pqtransformation}).

Let now $F_*(M)\wedge A\langle X\rangle_+$ be the $\I\times\I_*\langle X\rangle$-diagram in 
$\Sp^{\Sigma}$ defined by the levelwise smash products
$
F_*(M)(j)\wedge A\langle X\rangle(\theta)_+.
$
Then we have a natural transformation of $\Fin_*$-diagrams in $\mathcal D(\Sp^{\Sigma})$
$$
(\I\times \I_*\langle X\rangle, F_*(M)\wedge A\langle X\rangle_+)\to
(\I\times \I_*\langle X\rangle,F_*(A\langle X\rangle,M)),
$$
induced by the obvious natural transformation
\begin{equation}\label{F(M)map}
F_S(S(j_*),M(j_*))\wedge A\langle X\rangle(\theta)_+ \to 
F_S(S(j_*),M(j_*)\wedge A\langle X\rangle(\theta)_+)
\end{equation}
of $\I\times \I_*\langle X\rangle$-diagrams. Notice, that since homotopy colimits commute with smash products there is a canonical isomorphism
$$
\hocolim_{\I\times\I_*\langle X\rangle}F_*(M)\wedge A\langle X\rangle_+\cong
\Lambda_*(M)\wedge A_{h\I}'(X)_+.
$$
Evaluating the homotopy colimits of the above diagrams we therefore get natural transformations of $\Fin_*$-diagrams of $S$-modules
\begin{equation}\label{hocolimchain}
\Lambda'_X(T,M)\to \Lambda'_X(A,M)\leftarrow\Lambda_*(M)\wedge A_{h\I}'(X)_+.
\end{equation}
Furthermore, using Shipley's equivalence (\ref{shipleyequi}) we have the stable equivalence
\begin{equation}\label{ShipleyA}
\Lambda_*(M)\wedge A'_{h\I}(X)_+\simeq M\wedge A'_{h\I}(X)_+
\end{equation}
Finally we apply the same method used for $\Lambda(T,M)$ to extend the above functors to based simplicial sets. This completes our description of the chain of maps in Theorem \ref{Ispacehighthom}. 

\begin{lemma}
If $M$ is a commutative $T$-algebra, then the chain of maps in Theorem \ref{Ispacehighthom} is a chain of maps of $E_{\infty}$ $S$-algebras. 
\end{lemma}
\begin{proof}
Let $X$ be a (not necessarily finite) based set and let us write $\Lambda_X(A,M)$ for the homotopy colimit of the $\mathcal F_X$-diagram $Y\mapsto\Lambda'_Y(A,M)$ where $\mathcal F_X$ is as in Proposition \ref{EinftyLambda}. Modifying the construction used in the proof of Proposition \ref{EinftyLambda} to the functor $F_Y(A,M)$ we define an $\mathcal E$-action on this $\mathcal F_X$-diagram which is compatible with the 
$\mathcal E$-action on $\Lambda'(T,M)$ in the sense that there are commutative diagrams
$$
\begin{CD}
\mathcal E(n)_+\wedge \Lambda'_{Y_1}(T,M)\wedge\dots\wedge\Lambda'_{Y_n}(T,M)@>>> 
\Lambda'_Y(T,M)\\
@VVV @VVV\\
\mathcal E(n)_+\wedge \Lambda'_{Y_1}(A,M)\wedge\dots\wedge\Lambda'_{Y_n}(A,M)@>>> 
\Lambda'_Y(A,M).
\end{CD}
$$
Here $Y_1,\dots,Y_n$ denotes a family of finite based subsets of $X$ with union $Y$. By Corollary \ref{BEhocolim}
this implies that the left hand map in (\ref{hocolimchain}) extends to a map of $E_{\infty}$ $S$-algebras. A similar argument shows that the right hand map in (\ref{hocolimchain}) extends to an $E_{\infty}$ map. Finally, it follows from the observations in Example 
\ref{Einftyexample} that Shipley's definition of the stable equivalence in (\ref{ShipleyA}) extends to an $E_{\infty}$ map; we leave the details of this with the reader. 
\end{proof}

In order to show that these maps are in fact stable equivalences we first consider the case $M=T$.

\begin{lemma}\label{M=Tlemma}
The chain of maps in (\ref{hocolimchain}) are stable equivalences for $M=T$.
\end{lemma}
\begin{proof}
Since both functors in the theorem take $\I$-equivalences to stable equivalences of $S$-modules by Theorem \ref{Thominvariance} and the definition of $A_{h\I}'(X)$, we may without loss of generality assume that $A$ is convergent. Indeed, using the model structure on commutative 
$\I$-space monoids from Proposition \ref{commutativemodel}, any map $A\to BF$ may be factored as an $\I$-equivalence $A\to A'$ followed by a positive $\I$-fibration $A'\to BF$. Since $BF$ is convergent it follows from Lemma \ref{Ifibrationconvergence} that $A'$ is convergent as well.  As usual we may also assume that $f$ is $T$-good by applying the fibrant replacement functor $\Gamma$ if necessary, see Section \ref{symmetricthomsection}. For the right hand map in (\ref{hocolimchain}) we observe that by \cite{Sch2} the convergence assumption on $A$ implies that $T$ is connective and convergent as an $S$-module, hence it follows by standard arguments that the map in (\ref{F(M)map}) is a stable equivalence for each object in the indexing category. By the homotopy invariance of homotopy colimits, see e.g.\  \cite{Sh}, Lemma 4.1.5, this implies the result for this map. 
For the left hand side of (\ref{hocolimchain}) we first observe that the convergence of $A$ and $T$ ensure that 
B\" okstedt's approximation lemma applies such that the homotopy groups of the homotopy colimits are approximated by the homotopy groups of the individual terms, see  \cite{M}, Lemma 2.3.7. We claim that there exists an unbounded, non-decreasing sequence of integers $\lambda_n$ such that if each $\theta_x$ has cardinality greater than or equal to $n$, then the connectivity of the map
$$
\bigwedge_{x\in X}T_x(\theta_x)\to T(\theta_X)\wedge A\langle X\rangle_+,
$$
that is, the map (\ref{Mthommap}) for $M=T$, 
is greater than or equal to $\lambda_n+\Sigma_{x\in X}|\theta_x|$. Assuming this, it follows from B\"okstedt's approximation lemma that the natural transformation in 
(\ref{pqtransformation}) induces a levelwise equivalence of homotopy colimits. Since the Thom space functor takes products to smash products, the above map can be identified with the map of Thom spaces induced by the commutative diagram
\begin{equation}\label{Asplitting}
\begin{CD}
\displaystyle\prod_{x\in X}A(\theta_x)@>>> A(\theta_X)\times\displaystyle
\prod_{x\in \bar X}A(\theta_x)\\
@VVV @VVV\\
BF(\coprod_{x\in X}\theta_x)@>>> BF(\theta_X).
\end{CD}
\end{equation}
The first factor of the upper horizontal map is the composition
$$
\prod_{x\in X}A(\theta_x)\to A(\textstyle\coprod_{x\in X}\theta_x)\to A(\theta_X)
$$
and the second factor is the projection onto the components indexed by the points in 
$\bar X$. By the well-known connectivity properties of the Thom space functor 
(see e.g.\ \cite{Sch2}), 
it therefore suffices that there exists a sequence $\lambda_n$ such that the upper horizontal 
map in the diagram $(\ref{Asplitting})$ is at least 
$\lambda_n$-connected. Writing $A'$ for the product of the spaces 
$A(\theta_x)$ for $x\in\bar X$, the induced map of homotopy groups is represented by a triangular matrix of the form
$$
\left[\begin{array}{cc}
\pi_k(i)&*\\
0&\pi_k(\mathrm{id}_{A'}) 
\end{array}\right]\co \pi_kA(\theta_*)\times \pi_kA'\to \pi_kA(\theta_X)
\times \pi_kA',
$$
where $i$ is induced by the inclusion of the base point in $X$. Here we use that $A$ is convergent and grouplike to ensure that the homotopy groups $\pi_kA(\theta_x)$ are abelian groups also for $k=0$ and $k=1$. The convergence assumption implies that the connectivity of $i$ tends to infinity with the cardinality of $\theta_*$, hence the claim follows.  
\end{proof}

For the next lemma we recall from Section \ref{multiplicativesection} that if $T$ is a commutative $S$-algebra, then $\Lambda_*(T)$ is an $S$-algebra and  $\Lambda_X(T,T)$ is a 
$\Lambda_*(T)$-algebra. A similar argument shows that $\Lambda_*(M)$ is a (left and right) 
$\Lambda_*(T)$-module if $M$ is a $T$-module. Thus, if 
$\Lambda_*(T)^c\to \Lambda_*(T)$ denotes a cofibrant replacement of $\Lambda_*(T)$ as an $S$-algebra, then we may also view $\Lambda_X(T,T)$ and $\Lambda_*(M)$  as $\Lambda_*(T)^c$-modules. The following result is the analogue in our setting of the isomorphism in Remark
\ref{tensorremark}. 

\begin{lemma}\label{Lambdatensor}
Let $T$ be a well-based commutative $S$-algebra and $M$ a well-based $T$-module. Then there is a stable equivalence
$$
\Lambda_*(M)^c\wedge_{\Lambda_*(T)^c}\Lambda_X(T,T)\to \Lambda_X(T,M),
$$
where $\Lambda_*(T)^c$ is a cofibrant replacement of $\Lambda_*(T)$ and $\Lambda_*(M)^c$ is a cofibrant replacement of $\Lambda_*(M)$ as a right $\Lambda_*(T)^c$-module.
\end{lemma}
\begin{proof}
The map in question is induced by the composite map
$$
\Lambda_*(M)^c\wedge \Lambda_X(T,T)\to \Lambda_*(M)\wedge \Lambda_X(T,T)\to
\Lambda_X(T,M)
$$
where the last map is defined using the $T$-module structure on $M$. Smashing with the 
$\Lambda_*(T)^c$-module $\Lambda_*(M)^c$ preserves homotopy colimits and topological realization so that, by definition of $\Lambda_X(T,M)$, it suffices to prove the lemma when $X$ is a finite based set. For this we first show that the canonical map
$$
\bigwedge_{x\in X}\Lambda_x(T_x)^c\to
\bigwedge_{x\in X}\Lambda_x(T_x)\to \Lambda'_X(T,M)
$$
is a stable equivalence. Here $\Lambda_x(T_x)^c$ again denotes a cofibrant replacement and as usual $T_x=T$ for $x\neq *$  and $T_*=M$. Indeed, since both sides are homotopy invariant in the $T$ and $M$ variables by Proposition \ref{invarianceprop}, we may assume without loss of generality that $T$ and $M$ are cofibrant. It then follows from the definition of the stable equivalence in Proposition \ref{LambdaLoday} that there is a commutative diagram in the stable homotopy category
$$
\begin{CD}
\displaystyle{\bigwedge_{x\in X}}\Lambda_x(T_x)^c
@>>>\Lambda_X'(T,M)\\
@VV\sim V  @VV \sim V\\
\displaystyle{\bigwedge_{x\in X}}T_x@=\mathcal L_X(T,M)
\end{CD}
$$   
where the vertical maps are stable equivalences and the bottom horizontal map is an isomorphism. The upper horizontal map is therefore also an equivalence as claimed. Furthermore, in the case $M=T$ we may view the latter as an equivalence of $\Lambda_*(T)^c$-modules. Since smashing with the cofibrant module $\Lambda_*(M)^c$ preserves stable equivalence, we get a commutative diagram
$$
\begin{CD}
\Lambda_*(M)^c\wedge_{\Lambda_*(T)^c}\displaystyle{\bigwedge_{x\in X}}\Lambda_x(T)^c@>>>
\displaystyle{\bigwedge_{x\in X}}\Lambda_x(T_x)^c\\
@VV\sim V @VV\sim V\\
\Lambda_*(M)^c\wedge_{\Lambda_*(T)^c}\Lambda'_X(T,T)@>>>\Lambda'_X(T,M)
\end{CD}
$$  
in which the vertical maps are stable equivalences and the upper horizontal map is an isomorphism. This concludes the proof. 
\end{proof}

\medskip
\noindent\textit{Proof of Theorem \ref{Ispacehighthom}.}
It suffices to prove that the maps in (\ref{hocolimchain}) are stable equivalences for each finite based set $X$; the general case then follows from standard results on topological realization of (good) simplicial $S$-modules, see e.g.\ \cite{Sh}, Corollary 4.1.6. Let again $\Lambda_*(T)^c$ be a cofibrant replacement of $\Lambda_*(T)$ as an $S$-algebra and $\Lambda_*(M)^c$ a cofibrant replacement of $\Lambda_*(M)$ as a $\Lambda_*(T)^c$-module. We may then view the stable equivalences in Lemma \ref{M=Tlemma} as stable equivalences of 
$\Lambda_*(T)^c$-modules and we get a commutative diagram 
$$
\begin{CD}
\Lambda_*(M)^c\wedge_{\Lambda_*(T)^c}\Lambda_X'(T,T)@>\sim >> \Lambda_*(M)^c
\wedge_{\Lambda_*(T)^c}\Lambda_*(T)\wedge A'_{h\I}(X)_+\\
@VV\sim V @VV\sim V\\
\Lambda'_X(T,M) @>>> \Lambda_*(M)\wedge A'_{h\I}(X)_+,
\end{CD}
$$
where the bottom horizontal arrow represents the chain of maps in question and the vertical maps are the stable equivalences from Lemma \ref{Lambdatensor}. The chain of maps represented by the upper horizontal arrow are stable equivalences by Lemma \ref{M=Tlemma} and the fact 
that smash products with cofibrant modules preserve stable equivalences. Therefore the bottom arrow also represents a stable equivalence. \qed

\subsection{Homotopy orbits and Borel homology}\label{homotopyorbit}
Evaluating the functor $\Lambda(T)$ on the simplicial $n$-torus $(S^1)^n$ gives a model of the iterated topological Hochschild homology spectrum $\Th^{(n)}\!(T)$ and since 
$\Lambda_{(S^1)^n}(T)$ is the realization of an $n$-fold multi-cyclic set it inherits an action of the topological $n$-torus $\mathbb T^n$. In particular, $\Lambda_{S^1}(T)$ is a model of the topological Hochschild homology spectrum $\Th(T)$ with a canonical $\mathbb T$-action. Non-equivariantly, 
$\Lambda_{S^1}(T)$ is equivalent to B\"okstedt's model of $\Th(T)$ but this is far from true equivariantly: the finite cyclic groups $C_n$ act freely on $\Lambda_{S^1}(T)$ whereas it is the rich structure of the $C_n$-fixed point spectra in B\"okstedt's model that allows for the construction of topological cyclic homology \cite{BHM}. 
In future work we shall remedy this by introducing a refined version of the indexing categories 
$\I\langle X\rangle$ which in particular leads to a definition of higher topological cyclic homology. 
 In this paper we restrict attention to the associated homotopy orbit spectra and for this there is no problem since in general an equivariant map of symmetric spectra which is a non-equivariant stable equivalence induces a stable equivalence of homotopy orbits. We recall that for a group $G$ acting on an $S$-module $T$, the homotopy orbit spectrum is the $S$-module
 $T_{hG}=EG_+\wedge_G T$. For a $G$-space $B$, the homotopy orbit space is defined by 
 $B_{hG}=EG\times_G B$. Let $A$ be a commutative $\I$-space monoid and let
 $A_{h\I}$ be the associated  $\Gamma$-space introduced in Section \ref{Iinfinite}.
\begin{lemma}\label{freetorus}
If $A_{h\I}(S^0)$ is grouplike then there is a weak homotopy equivalence
$$
A_{h\I}((S^1)^n_+)_{hG}\xr{\sim} \Map(\mathbb T^n,A_{h\I}(S^n))_{hG}
$$ 
for any compact subgroup $G$ of $\mathbb T^n$. 
\end{lemma}
\begin{proof}
The space $A_{h\I}((S^1)^n_+)$ has a $\mathbb T^n$-action being the realization of an $n$-fold multi-cyclic space. Consider the composite map 
$$
\mathbb T^n\times A_{h\I}((S^1)^n_+)\to A_{h\I}((S^1)^n_+)\to A_{h\I}(S^n)
$$
where the first map is given by the $\mathbb T^n$-action and the second map is induced by the projection $(S^1_+)^{\wedge n}\to (S^1)^{\wedge n}$. The adjoint is a $\mathbb T^n$-equivariant map
$$
A_{h\I}((S^1)^n_+)\to \Map(\mathbb T^n,A_{h\I}(S^n))
$$ 
which we claim is a non-equivariant equivalence. The condition that $A_{h\I}(S^0)$ be grouplike implies that the $\Gamma$-space $A_{h\I}$ is very special in the sense of \cite{BF} and the result for $n=1$ is therefore an immediate consequence of \cite{Sch1}, Proposition 7.1. The general case can be deduced from this by viewing the functors 
$$
X\mapsto A_{h\I}(S^i\wedge (S^1_+)^{\wedge(n-i-1)}\wedge X)  
$$
as very special $\Gamma$-spaces and factoring the map in question as a composition of maps of the form
$$
\Map(\mathbb T^i,A_{h\I}(S^i\wedge (S^1_+)^{\wedge(n-i)}))\to
\Map(\mathbb T^{i+1},A_{h\I}(S^{i+1}\wedge (S^1_+)^{n-i-1})),
$$ 
each of which is an equivalence. 
\end{proof}

\medskip
\noindent\textit{Proof of Corollary \ref{homotopyorbitcorollary}.} It is well-known that the integral Eilenberg-Mac Lane spectrum $H\mathbb Z$ is an algebra over any oriented Thom $S$-algebra; an explicit construction of such an algebra structure can be found in \cite{Sch2}. Applying Theorem \ref{Ispacehighthom} to the based simplicial set 
$(S^1)_+^n$ we therefore get a stable equivalence
$$
H \mathbb Z \wedge \Lambda_{(S^1)^n}(T)\simeq \Lambda_{(S^1)^n_+}(T,H\mathbb Z)\simeq H\mathbb Z\wedge A_{h\I}(\mathbb T^n_+)_+.
$$
By construction this is the topological realization of a chain of maps of $n$-fold multi-cyclic $S$-modules, hence is $\mathbb T^n$-equivariant. The action on $H\mathbb Z$ is trivial and the result now follows by passing to homotopy orbits and using Lemma  \ref{freetorus}.\qed

\medskip  
\appendix
\section{Multiplicative properties of homotopy colimits}\label{appendix}
We here discuss the general multiplicative properties of homotopy colimits of $S$-modules and we develop the framework necessary for the definition of the $E_{\infty}$ structure on $\Lambda_X(T,M)$ in Section \ref{multiplicativesection} when  $M$ is a commutative 
$T$-algebra. The main result is Corollary \ref{BEhocolim} which states that in a precise sense homotopy colimits take $E_{\infty}$ diagrams to $E_{\infty}$ $S$-algebras. A similar result holds for diagrams of spaces as we explain in Section \ref{spacehocolim}.
\subsection{Symmetric monoidal structures}\label{symmetricmonoidal}
In order to fix terminology we first give a brief review of symmetric monoidal functors, following Mac Lane \cite{Mac}. We say that a monoidal category 
$(\mathcal A,\Box,1_{\mathcal A})$ is \emph{strict monoidal} if the monoidal product is strictly associative and unital, that is,
$$
(a\Box b)\Box c=a\Box(b\Box c),\quad 1_{\mathcal A}\Box a=a,\quad 
a\Box 1_{\mathcal A}=a,
$$  
for all objects $a$, $b$, and $c$ in $\mathcal A$. A \emph{permutative category} is a symmetric monoidal category which is strict monoidal and a \emph{strict symmetric monoidal category} is a permutative category such that the symmetry isomorphisms 
$\tau\co a\Box b\to b\Box a$ are identities. Given a permutative category 
$\mathcal A$, a \emph{monoidal functor} $F\co \mathcal A\to \Sp^{\Sigma}$ is a functor equipped with a unit 
$\eta\co S\to F(1_{\mathcal A})$ and a natural transformation of functors on $\mathcal A\times \mathcal A$,
$$
\mu_{a,b}\co F(a)\wedge F(b)\to F(a\Box b),
$$  
which satisfies that usual associativity and unitality conditions (this is what is sometimes called lax monoidal). The functor $F$ is \emph{symmetric monoidal} if in addition the diagrams
$$
\begin{CD}
F(a)\wedge F(b)@>\mu_{a,b} >> F(a\Box b)\\
@VVV @VV F(\tau) V\\
F(b)\wedge F(a)@> \mu_{b,a}>> F(b\Box a)
\end{CD}
$$
are commutative, where on the left the vertical map is the symmetry isomorphism in 
$\Sp^{\Sigma}$. We say that $F$ is \emph{strong (symmetric) monoidal} if the structure maps 
$\eta$ and $\mu_{a,b}$ are isomorphisms and \emph{strict (symmetric) monoidal} if these maps are identities. Given symmetric monoidal functors $F, G\co \mathcal A\to \Sp^{\Sigma}$, a \emph{monoidal natural transformation} $\phi\co F\to G$ is a natural transformation making the diagrams
$$
\begin{CD}
F(a)\wedge F(b)@>>> F(a\Box b)\\
@VV \phi_a\wedge\phi_b V @VV \phi_{a\Box b} V\\
G(a)\wedge G(b)@>>> G(a\Box b)
\end{CD}
\qquad\text{ and }\qquad
\begin{CD}
S@>>> F(1_{\mathcal A})\\
@| @VV\phi_{1_{\mathcal A}}V\\
S @>>> G(1_{\mathcal A})
\end{CD}
$$   
commutative. Let now $\mathcal D(\Sp^{\Sigma})$ be the category of diagrams in 
$\Sp^{\Sigma}$ as defined in Section \ref{hocolimsection} and consider the symmetric monoidal structure 
which to a pair of objects $(\mathcal A,F)$ and $(\mathcal A',F')$ associates the object 
$(\mathcal A\times \mathcal A,'F\wedge F')$, where $F\wedge F'$ denotes the functor
$$
F\wedge F'\co \mathcal A\times \mathcal A'\to \Sp^{\Sigma}\times \Sp^{\Sigma}\xr{\wedge}\Sp^{\Sigma}. 
$$
The object $(\mathbf 1,S)$ is the unit. A monoid structure on an object $(\mathcal A,F)$ 
amounts to 
\begin{itemize}
\item
a strict monoidal structure on $\mathcal A$, and
\item
a monoidal structure on the functor $F\co\mathcal A\to\Sp^{\Sigma}$.
\end{itemize}
Such a monoid is commutative if and only if $\mathcal A$ is strictly symmetric and the functor $F$ is symmetric monoidal. Given monoids $(\mathcal A,F)$ and $(\mathcal B,G)$, a morphism 
$$
(\varepsilon,\phi)\co (\mathcal A,F)\to (\mathcal B,G) 
$$
in $\mathcal D(\Sp^{\Sigma})$ is a monoid morphism if and only if 
\begin{itemize}
\item
the functor $\varepsilon\co \mathcal A\to\mathcal B$ is strict monoidal, and
\item
the natural transformation $\phi\co F\to G\circ\varepsilon$ is monoidal.
\end{itemize}
The following lemma follows immediately from the fact that smash products of $S$-modules are distributive with respect to wedge products.
\begin{lemma} 
The homotopy colimit functor $\hocolim\co\mathcal D(\Sp^{\Sigma})\to\Sp^{\Sigma}$ is strong symmetric monoidal.\qed
\end{lemma}
Thus, for each pair of objects $(\mathcal A,F)$ and $(\mathcal B,G)$ there is a natural isomorphism of $S$-modules 
$$
\hocolim_{\mathcal A}F\wedge \hocolim_{\mathcal A'}F'\cong \hocolim_{\mathcal A\times \mathcal A'}F\wedge F'.
$$
As a consequence, hocolim takes monoids in $\mathcal D(\Sp^{\Sigma})$ to $S$-algebras and commutative monoids to commutative $S$-algebras. However, the condition for 
an object in $\mathcal D(\Sp^{\Sigma})$ to be a commutative monoid is very strong and such objects occur rarely in the applications. In the next section we shall formulate a more useful symmetry condition on objects in $\mathcal D(\Sp^{\Sigma})$ which implies that their homotopy colimits are $E_{\infty}$ $S$-algebras.

\subsection{Homotopy colimits and $E_{\infty}$ structures}\label{hocolimEinfty}
We first recall that if $(\mathcal A,\Box,1_{\mathcal A})$ is a small permutative category, then the category $\mathcal A\Sp^{\Sigma}$ of functors $F\co \mathcal A\to \Sp^{\Sigma}$ inherits a symmetric monoidal product which to a pair of objects $F$ and $G$ associates the object $F\triangle G$ defined by
$$
F\triangle G(a)=\colim_{a_1\Box a_2\to a}F(a_1)\wedge F(a_2).
$$ 
This is the usual Kan extension of the $\mathcal A^2$-diagram $F\wedge G$ along the monoidal structure map $\Box\co \mathcal A^2\to \mathcal A$, see e.g. \cite{MMSS}. Given an operad 
$\mathcal C$ in the sense of \cite{May} we as usual get an induced monad $C$ on 
$\mathcal A\Sp^{\Sigma}$ defined by
$$
C(F)=\bigvee_{n\geq 0}\mathcal C(n)_+\wedge_{\Sigma_n}F^{\triangle n}
$$ 
and we say that $\mathcal C$ acts on an object $F$ in $\mathcal A\Sp^{\Sigma}$ if $F$ is an algebra for this monad. By the universal property of the $\triangle$ product such an action amounts to a sequence of natural transformations
$$
\theta_n\co\mathcal C(n)_+\wedge F(a_1)\wedge\dots\wedge F(a_n)\to F(a_1\Box\dots\Box a_n)
$$ 
where for each $n$ we view both sides as functors from $\mathcal A^n$ to $\Sp^{\Sigma}$. 
These maps are required to satisfy the usual unitality, associativity and equivariance conditions, see \cite{May}, Lemma 1.4. In particular, the equivariance condition amounts to the commutativity of the diagram
$$
\begin{CD}
\mathcal C(n)_+\wedge F(a_1)\wedge\dots\wedge F(a_n)@>\theta_n\circ(\sigma_+\wedge\id)>>
F(a_1\Box\cdots\Box a_n)\\
@VV \id_+\wedge \sigma V @VV F(\sigma) V\\
\mathcal C(n)_+\wedge F(a_{\sigma^{-1}(1)})\wedge \dots\wedge F(a_{\sigma^{-1}(n)})
@>\theta_n >>  F(a_{\sigma^{-1}(1)}\Box\cdots\Box a_{\sigma^{-1}(n)})
\end{CD}
$$
for all elements $\sigma$ in $\Sigma_n$. Here $\sigma$ permutes the factors on the left hand side of the diagram and $F(\sigma)$ is the map obtained by applying $F$ to the canonical isomorphism determined by $\sigma$. We define a \emph{permutative $\mathcal C$-diagram} in $\Sp^\Sigma$ to be a pair $(\mathcal A, F)$ where $\mathcal A$ is a permutative small category
and  $F$ is an object with $\mathcal C$-action in $\mathcal A\Sp^{\Sigma}$. 

\begin{example}\label{Einftyexample}  
For $\mathcal C$ the commutativity operad with $\mathcal C(n)=*$, a permutative $\mathcal C$-diagram $(\mathcal A,F)$ is the same thing as a permutative category $\mathcal A$ together with a symmetric monoidal functor $F\co\mathcal A\to\Sp^{\Sigma}$. 
For example, one can check that B\"okstedt's functor $F_X(T,M)$ from Section \ref{Bsection} is symmetric monoidal if $M$ is a commutative $T$-algebra.
\end{example}

Let $\mathcal P\mathcal D(\Sp^{\Sigma},\mathcal C)$ be the category of permutative $\mathcal C$-diagrams in $\Sp^{\Sigma}$ in which a morphism
$(\varepsilon,\phi)$ from $(\mathcal A,F)$ to $(\mathcal B,G)$
is a morphism in $\mathcal D(\Sp^{\Sigma})$ such that
\begin{itemize}
\item
the functor $\varepsilon\co \mathcal A\to\mathcal B$ is strict symmetric monoidal, and
\item
the natural transformation $\phi\co F\to G\circ\varepsilon$ is a $\mathcal C$-map in 
$\mathcal A\Sp^{\Sigma}$.
\end{itemize}
We claim that the homotopy colimit functor takes an object in 
$\mathcal P\mathcal D(\Sp^{\Sigma},\mathcal C)$ to an $S$-module with a canonical action of the product operad $\mathcal E\times\mathcal C$ where $\mathcal E$ denotes the 
Barratt-Eccles operad. Recall that the $n$th space $\mathcal E(n)$ of this operad is the classifying space of the translation category $\tilde\Sigma_n$. The latter category has as its objects the elements of $\Sigma_n$ and a morphism $\rho\co \sigma\to\tau$ is an element $\rho\in \Sigma_n$ such that $\rho\sigma=\tau$. Since a morphism in $\tilde\Sigma_n$ is uniquely determined by its domain and target, its nerve may be identified with the simplicial set   
\begin{equation}\label{BEsimplicial}
E_{\bullet}\Sigma_n\co[k]\mapsto\Map([k],\Sigma_n)=\Sigma^{k+1}.
\end{equation}
Let $\Sp^{\Sigma}[\mathcal E\times\mathcal C]$ be the category of $S$-modules with 
($\mathcal E\times \mathcal C$)-action.

\begin{proposition}\label{hocolimEinftyprop}
The homotopy colimit functor induces a functor
$$
\hocolim\co\mathcal P\mathcal D(\Sp^{\Sigma},\mathcal C)\to\Sp^{\Sigma}
[\mathcal E\times \mathcal C].
$$
\end{proposition}

Before proving this result we first recall from \cite{May1}, Theorem 4.9, that the classifying space of a permutative category $(\mathcal A,\Box,1_{\mathcal A})$ has a canonical 
$\mathcal E$-action. Indeed, the symmetric monoidal structure of $\mathcal A$ gives rise to functors
\begin{equation}\label{BEaction}
\tilde\Sigma_n\times\mathcal A^n\to\mathcal A,\quad (\sigma,(a_1,\dots,a_n))
\mapsto a_{\sigma^{-1}(1)}\Box\dots\Box a_{\sigma^{-1}(n)},
\end{equation}
inducing a sequence of simplicial maps 
$$
B_{\bullet}\tilde\Sigma_n\times B_{\bullet}\mathcal A^n\to B_{\bullet}\mathcal A.
$$
In simplicial degree $k$ this takes an element
$$
(\sigma_0,\dots,\sigma_k)\in \Sigma_n^{k+1},\quad (a_0^i\leftarrow\dots\leftarrow a_k^i)\in B_k\mathcal A, \text{ for $i=1,\dots,n$},
$$ 
to the $k$-simplex
\begin{equation}\label{ksimplex}
a_0^{\sigma_0^{-1}(1)}\Box\dots\Box a_0^{\sigma_0^{-1}(n)}\leftarrow\dots\leftarrow
a_k^{\sigma_k^{-1}(1)}\Box\dots\Box a_k^{\sigma_k^{-1}(n)}.
\end{equation}
It follows from the axioms for a permutative category that this defines an action of 
$\mathcal E$ on $B\mathcal A$ after topological realization.

\medskip
\noindent\textit{Proof of Proposition \ref{hocolimEinftyprop}.}
Let $(\mathcal A,F)$ be an object of $\mathcal P\mathcal D(\Sp^{\Sigma},\mathcal C)$. 
An action of $\mathcal E\times \mathcal C$ on $\hocolim_{\mathcal A}F$ then amounts to a sequence of maps of $S$-modules
$$ 
(\mathcal E(n)\times \mathcal C(n))_+\wedge(\hocolim_{\mathcal A}F)^{\wedge n}
\to \hocolim_{\mathcal A}F
$$
that satisfy the usual unitality, associativity and equivariance conditions. Using that hocolim is strong symmetric monoidal and that topological realization preserves smash products, it suffices to define maps of simplicial $S$-modules
$$
(\mathcal C(n)\times\Sigma_n^{k+1})_+\wedge\bigvee_{\left\{\substack{a_0^i\leftarrow\dots\leftarrow a_k^i,\\
\text{$i=1,\dots, n$}}\right\}}
 F(a_k^1)\wedge\dots\wedge F(a_k^n)\to
\bigvee_{\left\{b_0\leftarrow\dots\leftarrow b_k\right\}} F(b_k)
$$
satisfying the analogous conditions. Keeping an element $(\sigma_0,\dots,\sigma_k)$ of 
$\Sigma^{k+1}_n$ fixed, we do this using the maps
$$
\mathcal C(n)_+\wedge F(a_k^1)\wedge \dots\wedge F(a_k^n)\xr{\theta_n} 
F(a_k^1\Box\dots\Box a_k^n)
\to F(a_k^{\sigma^{-1}_k(1)}\Box\dots\Box a_k^{\sigma_k^{-1}(n)})
$$
where $\theta_n$ denotes the $\mathcal C$-action on $F$ and the second arrow is obtained by applying $F$ to the canonical morphism induced by the permutation $\sigma_k$. 
Thus, each of the wedge components of the domain is mapped to the corresponding wedge component of the target indexed by the $k$-simplex in (\ref{ksimplex}). It is easy to verify that this is a simplicial map and that the required relations hold.
Furthermore, a morphism in $\mathcal P\mathcal D(\Sp^{\Sigma},\mathcal C)$ gives rise to a map of $S$-modules which is clearly compatible with the 
($\mathcal E\times \mathcal C$)-action. \qed

\medskip
Notice that if $\mathcal C$ is augmented over $\mathcal E$ in the sense that there is a map of operads $\mathcal C\to\mathcal E$, then we may pull an ($\mathcal E\times \mathcal C$)-action  back to a $\mathcal C$-action via the diagonal $\mathcal C\to\mathcal E\times\mathcal C$. In particular this applies if $\mathcal C$ is the operad $\mathcal E$ itself and gives the following corollary.

\begin{corollary}\label{BEhocolim}
The homotopy colimit functor induces a functor
$$
\hocolim\co\mathcal P\mathcal D(\Sp^{\Sigma},\mathcal E)\to\Sp^{\Sigma}[\mathcal E]
\eqno\qed
$$
\end{corollary}

\subsection{Homotopy colimits of spaces}\label{spacehocolim}
Everything done for diagrams of $S$-modules above have obvious analogues for diagrams of (unbased) spaces. Thus, we have the symmetric monoidal category $\mathcal D(\mathcal U)$ of diagrams in $\mathcal U$ and the Bousfield-Kan homotopy colimit functor
$\mathcal D(\mathcal U)\to\mathcal U$
is defined by replacing the wedge products in (\ref{hocolimdefinition}) by coproducts of spaces. This is again a strong symmetric monoidal functor. Given an operad $\mathcal C$ we have the category $\mathcal P\mathcal D(\mathcal U,\mathcal C)$ of permutative diagrams with $\mathcal C$-action and the following analogues of Proposition \ref{hocolimEinftyprop} and Corollary \ref{BEhocolim}. 

\begin{proposition}
The homotopy colimit functor induces a functor
$$
\hocolim\co\mathcal P\mathcal D(\U,\mathcal C)\to \mathcal U[\mathcal E\times \mathcal C].
\eqno\qed
$$
\end{proposition}

\begin{corollary}\label{spaceEinfty}
The homotopy colimit functor induces a functor
$$
\hocolim\co\mathcal P\mathcal D(\mathcal U,\mathcal E)\to\mathcal U[\mathcal E].
\eqno\qed
$$
\end{corollary}
This applies in particular to any symmetric monoidal functor 
$F\co\mathcal A\to\mathcal U$ defined on a permutative category $\mathcal A$; hence 
$\hocolim_{\mathcal A}F$ then has a canonical $E_{\infty}$ action.

\end{document}